\documentclass[11pt]{amsart}
\usepackage{amssymb,amsmath,amsfonts,amscd,euscript}
\usepackage{textfit} 

\newcommand{\nc}{\newcommand}

\numberwithin{equation}{section}
\newtheorem{thm}{Theorem}[section]
\newtheorem*{thm*}{Theorem}
\newtheorem{prop}[thm]{Proposition}
\newtheorem{lem}[thm]{Lemma}
\newtheorem{exam}[thm]{Example}
\newtheorem{cor}[thm]{Corollary}
\newtheorem*{cor*}{Corollary}
\theoremstyle{remark}
\newtheorem{rem}[thm]{Remark}
\newtheorem{definition}[thm]{Definition}
\newtheorem{dfn}[thm]{Definition}

\newcommand{\semi}{{\,\rule[.1pt]{.4pt}{5.3pt}\hskip-1.9pt\times}}

\newcount\cols {\catcode`,=\active\catcode`|=\active
 \gdef\Young(#1){\hbox{$\vcenter
 {\mathcode`,="8000\mathcode`|="8000
  \def,{\global\advance\cols by 1 &}%
  \def|{\cr
        \multispan{\the\cols}\hrulefill\cr
        &\global\cols=2 }%
  \offinterlineskip\everycr{}\tabskip=0pt
  \dimen0=\ht\strutbox \advance\dimen0 by \dp\strutbox
  \halign
   {\vrule height \ht\strutbox depth \dp\strutbox##
    &&\hbox to \dimen0{\hss$##$\hss}\vrule\cr
    \noalign{\hrule}&\global\cols=2 #1\crcr
    \multispan{\the\cols}\hrulefill\cr%
   }
 }$}}
 \gdef\Skew(#1:#2){\hbox{$\vcenter
 {\mathcode`,="8000\mathcode`|="8000
  \dimen0=\ht\strutbox \advance\dimen0 by \dp\strutbox
  \def\boxbeg{\vbox
    \bgroup\hrule\kern-0.4pt\hbox to\dimen0\bgroup\strut\vrule\hss$}%
  \def\boxend{$\hss\egroup\hrule\egroup}%
  \def,{\boxend\boxbeg}%
  \def|##1:{\boxend\vrule\egroup\nointerlineskip\kern-0.4pt
    \moveright##1\dimen0\hbox\bgroup\boxbeg}%
  \def\\##1\\##2:{\boxend\vrule\egroup\nointerlineskip\kern-0.4pt
    \kern ##1\dimen0\moveright##2\dimen0\hbox\bgroup\boxbeg}%
  \moveright#1\dimen0\hbox\bgroup\boxbeg#2\boxend\vrule\egroup
 }$}} }

\def\smallsquares
 {\textfont0=\scriptfont0 \scriptfont0=\scriptscriptfont0
  \textfont1=\scriptfont1 \scriptfont1=\scriptscriptfont1
  \setbox0=\hbox{$($}
  \setbox\strutbox=\hbox{\vrule width 0pt height\ht0 depth\dp0 }
 }
\nc{\gl}{\mathfrak{gl}}
\nc{\GL}{\mathfrak{GL}}
\nc{\g}{\mathfrak{g}}
\nc{\gh}{\widehat\g}
\nc{\h}{\mathfrak{h}}
\nc{\n}{\mathfrak{n}}
\nc{\la}{\lambda}
\nc{\C}{\mathbb C }
\nc{\D}{\mathbb D }
\nc{\Z}{\mathbb Z }
\nc{\N}{\mathbb N }
\nc{\R}{\mathbb R }
\nc{\Q}{\mathbb Q }
\nc{\al}{\alpha }
\nc{\bs}{{\bf s}}
\nc{\bbs}{\bar{\bf s}}
\nc{\bt}{{\bf t}}
\nc{\br}{{\bf r}}
\nc{\bp}{{\bf p}}
\nc{\bm}{{\bf m}}
\nc{\ba}{\bar{\alpha} }
\nc{\bj}{{\bf j}}
\nc{\f}{{\mathfrak{f}}}
\nc{\om}{\omega}
\nc{\ta}{\theta}
\nc{\ve}{\varepsilon}
\nc{\ch}{{\mathop {\rm ch}}}
\nc{\Tr}{{\mathop {\rm Tr}\,}}
\nc{\Id}{{\mathop {\rm Id}}}
\nc{\ad}{{\mathop {\rm ad}}}
\nc{\charc}{{\mathop {\rm char}}}
\nc{\grad}{{\mathop {\rm grad\,}}}
\nc{\bra}{\langle}
\nc{\ket}{\rangle}
\nc{\x}{{\bf x}}
\nc{\pa}{\partial}
\nc{\ld}{\ldots}
\nc{\cd}{\cdots}
\nc{\hk}{\hookrightarrow}
\nc{\T}{\otimes}
\nc{\mo}{{\scaletoheight{0.08cm}{mon}}}
\nc{\ord}{\succ_{\mo}}
\nc{\bk}{{\bf k}}
\nc{\bx}{{\bf x}}
\newcommand{\bea}{\begin{equation}}
\newcommand{\ena}{\end{equation}}
\newcommand{\be}{\begin{equation*}}
\newcommand{\en}{\end{equation*}}
\nc{\gr}{\mathrm{gr}}
\nc{\ov}{\overline}

\nc{\msl}{\mathfrak{sl}}
\nc{\mgl}{\mathfrak{gl}}
\nc{\U}{\hbox{\rm U}}
\nc{\V}{\EuScript V}
\nc{\msp}{\mathfrak{sp}}

\newcommand{\bc}{{\mathbb C}}
\newcommand{\bz}{{\mathbb Z}}

\newcommand{\fg}{{\mathfrak g}}
\newcommand{\fb}{{\mathfrak b}}

\newcommand{\fn}{{\mathfrak n}}

\newcommand{\lam}{\lambda}

\newcommand{\ol}{\overline}

\def\gr{\operatorname{gr}}

\newcommand{\Fl}{\EuScript{F}}
\begin{document}

\title[PBW--filtration over $\bz$ and bases]
{PBW--filtration over $\bz$ and compatible bases for $V_\bz(\la)$ in type ${\tt A}_n$ and ${\tt C}_n$}

\author{Evgeny Feigin, Ghislain Fourier and Peter Littelmann}
\address{Evgeny Feigin:\newline
Department of Mathematics,\newline
National Research University Higher School of Economics,\newline
Vavilova str. 7, 117312, Moscow, Russia,\newline
{\it and }\newline
Tamm Theory Division, Lebedev Physics Institute}
\email{evgfeig@gmail.com}
\address{Ghislain Fourier:\newline
Mathematisches Institut, Universit\"at zu K\"oln,\newline
Weyertal 86-90, D-50931 K\"oln,Germany
}
\email{gfourier@math.uni-koeln.de}
\address{Peter Littelmann:\newline
Mathematisches Institut, Universit\"at zu K\"oln,\newline
Weyertal 86-90, D-50931 K\"oln,Germany
}
\email{littelma@math.uni-koeln.de}

\begin{abstract}
We study the PBW-filtration on the highest weight representations $V(\la)$ of
the Lie algebras of type ${\tt A}_{n}$ and ${\tt C}_{n}$. This
filtration is induced by the standard degree filtration on $\U(\fn^-)$. 
In previous papers, the authors studied the filtration and the associated graded algebras and
modules over the complex numbers. The aim of this paper is to present a proof of the results
which holds over the integers and hence makes the whole construction available 
over any field.
\end{abstract}
\maketitle
\section*{Introduction}
Let $\fg$ be a simple complex Lie algebra, we fix a maximal torus
$\h$ and a Borel subalgebra $\fb=\h\oplus \n^+$. Denote by $R$
the set of roots and let $P$ be the integral weight lattice. Corresponding to the 
choice of $\fb$, let $R^+$ be the set of positive roots and let $P^+$ be the
monoid of dominant weights. 

For $\la\in P^+$ let $V(\lam)$ be the finite dimensional irreducible representation of highest weight
$\la$ and denote by $M(\la)$ the Verma module corresponding to the same highest weight. 
For a Lie algebra ${\mathfrak a}$ denote by $\U({\mathfrak a})$ its enveloping algebra.
Fix a highest weight vector $m_\la\in M(\la)$. The linear map 
$$
\U(\fn^-)\rightarrow M(\la), {\mathbf n}\mapsto  {\mathbf n}m_\la
$$ 
is an isomorphism of complex vector spaces.
The degree filtration on $U(\fn^-)$:
$$
\U(\n^-)_0=\bc 1, \quad \U(\n^-)_s=\mathrm{span}\{1,x_1\dots x_l:\ x_i\in\n^-, l \le s \}\text{\ for $s\ge 1$},
$$ 
induces via the isomorphism above a natural $\fb$-stable filtration on $M(\la)$:
$$
M(\la)_s= \U(\n^-)_s m_\la\quad\text{for\ }s\ge 0.
$$ 
Set $\U(\n^-)_{-1}=M(\la)_{-1}=0$, then the associated $q$-character 
$$\charc_q M(\la):=\sum_{s\ge 0} \charc (M(\la)_s/M(\la)_{s-1}) q^s$$ 
has a very simple form:
$$
\charc_q M(\la)=e^\la\frac{1}{\prod_{\beta\in R^+} (1-qe^{-\beta})}.
$$
This is obvious by the fact that the associated graded module $M(\la)^a=\bigoplus_{s\ge 0}M(\la)_s/M(\la)_{s-1}$
is a free module over the associated graded algebra $S(\fn^-)=\grad U(\fn^-)$.

In contrast, the situation becomes rather complicated if one replaces $M(\la)$ by its finite dimensional quotient $V(\la)$.
Again this module has an induced $\fb$-stable filtration $V(\la)_s=\U(\n^-)_s v_\la$, called the {\it Poincar\'e-Brikhoff-Witt -filtration},
or, for short, just the {PBW}-filtration. 
The associated graded module $V(\la)^a=\bigoplus_{s\ge 0}V(\la)_s/V(\la)_{s-1}$ is a $U(\fb)$-module as well as a 
$S(\fn^-)$-module. A general closed formula for the $q$-character 
$$
\charc_q V(\la):=\sum_{s\ge 0} \charc (V(\la)_s/V(\la)_{s-1}) q^s
$$
is not known, partial combinatorial answers can be found in \cite{FFoL1,FFoL2}, more geometric
interpretations can be found in \cite{FF,FFiL}. Another natural (and, at least in the general case, open) question is about the structure
of $V(\la)^a$ as a cyclic $S(\fn^-)$-module, generated by the image of the highest weight vector.

The aim of this paper is to present a proof of the results in \cite{FFoL1,FFoL2} which holds over the integers
and hence makes the whole construction available over any field. More precisely, for $\fg$ 
of type ${\tt A}_n$ or type ${\tt C}_n$ we want   
\begin{itemize}
\item to describe $ V^a_\bz(\la)$ as a cyclic $S_\bz(\n^-)$-module, i.e. describe the ideal
$I_\bz(\la)\hk S_\bz(\n^-)$ such that $ V^a_\bz(\la)\simeq S_\bz(\n^-)/I_\bz(\la)$;
\item to find a basis of $V^a_\bz(\la)$, in particular, show that $V^a_\bz(\la)$ is torsion free;
\item to get a (characteristic free) combinatorial graded character formula for $V^a_\bz(\la)$.
\end{itemize}
As a last remark we would like to point out that one should not confuse the PBW-filtration
(discussed in this paper) neither with the Brylinski-Kostant filtration \cite{Br} (BK-filtration for short) on the weight spaces
induced by a principle ${\mathfrak{sl}}_2$-triple $(e,h,f)$, nor with the right Brylinski-Kostant filtration 
discussed in \cite{HJ}. As an example, consider the case $\fg$ of type ${\tt B}_2$ and $\la=\omega_1+2\omega_2$. 
In the table below we list for some weights the Poincar\'e polynomial of the associated graded
weight space. For the left and right Brylinski-Kostant filtration, the polynomials have been taken from
\cite{HJ}, for the PBW-filtration the polynomials have been calculated using Theorem~\ref{spanCn}
(${\tt B}_2={\tt C}_2$).
\begin{table}[htdp]
\caption{Examples for the Poincar\'e polynomial of the associated graded weight spaces in $V(\la)$, $\la=\omega_1+2\omega_2$,
$\fg$ of type ${\tt B}_2$, enumeration as in \cite{B}.}
\begin{center}
\begin{tabular}{|c|c|c|c|c|}
\hline
weight & $\la-\alpha_1-3\alpha_2$& $\la-2\alpha_1-2\alpha_2$ & $\la-2\alpha_1-3\alpha_2$ & $\la-2\alpha_1-4\alpha_2$ \\
\hline
PBW & $q^3+q^2$ &  $q^3+2q^2$ &  $2q^3+q^2$ &  $q^4+q^3+q^2$  \\
\hline
BK & $q^4+q^3$ & $q^4+q^3+q^2$ & $q^5+q^4+q^3$ & $q^6+q^5+q^4$ \\
\hline
right BK & $q^4+q^2$ &  $q^4+q^3+q^2$ & $q^5+q^4+q^3$ & $q^6+q^5+q^4$ \\
\hline
\end{tabular}
\end{center}
\label{default}
\end{table}%

\section{The setup over the complex numbers: definitions and notation}
Let $\fg$ be a simple Lie algebra. We fix a Cartan subalgebra
$\h$ and a Borel subalgebra $\fb=\h\oplus \n^+$.
Let $R^+$ be the set of positive roots corresponding to the choice of $\fb$
and let $\al_i$, $\omega_i$ $i=1,\dots,n$ be the simple roots and the fundamental weights.
The height $ht(\beta)$ of a positive root is the sum of the coefficients
of the expression of $\beta$ as a sum of simple roots.

Let $G$ be the simple, simply connected algebraic group such that $\text{Lie}\, G=\fg$.
Fix a maximal torus $T\subset G$ and a Borel subgroup $B\supset T$ such that
$\text{Lie}\, B=\h\oplus \n^+$and $\text{Lie}\,T=\h$.
Denote by $N^-$ the unipotent radical of the opposite Borel subgroup.

Let $\g=\n^+\oplus\h\oplus \n^-$ be the Cartan decomposition.
Consider the increasing degree  filtration on
the universal enveloping algebra of $\U(\n^-)$:
\begin{equation}\label{dfsetup}
\U(\n^-)_s=\mathrm{span}\{1,x_1\dots x_l:\ x_i\in\n^-, l \le s \},
\end{equation}
for example, $\U(\n^-)_0=\C \cdot 1$, $\U(\n^-)_1=\C \cdot 1+\fn^-$, and so on. The associated graded algebra is the
symmetric algebra $S(\n^-)$ over $\n^-$.

For a dominant integral weight  $\la$ let $\Psi:G \rightarrow \text{GL}(V(\la))$
and $\psi:\fg\rightarrow \text{End}(V(\la))$
be the corresponding irreducible representations. Fix
a highest weight vector $v_\la$.
Since $V(\la)=\U(\n^-)v_\la$, the filtration in \eqref{dfsetup} induces an increasing  filtration $V(\la)_s$ on $V(\la)$:
\[
V(\la)_s=\U(\n^-)_s v_\la.
\]
\begin{definition}
We call this filtration the {\it PBW-filtration of $V(\la)$} and we denote the associated graded
space by $V^a(\la)$. 
\end{definition}
Let $\n^-_s=\sum_{ht(\beta)\ge s}\n^-_{\beta}\subseteq\fn^-$ be the Lie subalgebra formed by the root subspaces
corresponding to roots of height at least $s$. In fact, $\fn^-_{s}\subset \fn^-$ is an ideal,
and the associated graded algebra $\n^{-,a}=\bigoplus_{s\ge 1} \fn^-_{s}/\fn^-_{s+1}$
is an abelian Lie algebra. We make $\n^{-,a}$ into a $B$- as well as a $\fb$-module by identifying 
the vector space $\fn^{-,a}$ with the quotient space $\g/\fb$, which is a $B$- respectively $\fb$-module via the induced 
adjoint action $\ov{ad}:B\rightarrow GL(\fg/\fb)$.

\begin{definition}
Denote by $\g^a$  the Lie algebra $g^a= \fb\oplus \fn^{-,a}$, 
where $\fn^{-,a}$ is an abelian ideal in $\fg^a$ and $\fb$
acts on $\fn^{-,a}$ via the induced adjoint action described above. 
\end{definition}
For a positive root $\beta$ let $U_{-\beta}\subset G$ be the closed root subgroup corresponding to the root $-\beta$.
Denote by $x_{-\beta}:{\mathbb G}_{a,\beta}\rightarrow U_{-\beta}$ a fixed isomorphism of the root subgroup
with the additive group ${\mathbb G}_{a}$. We add the root as an index to indicate that this copy ${\mathbb G}_{a,\beta}$
of the additive group is related to  $U_{-\beta}$.

The group $N^-$ admits a filtration by a sequence of normal subgroups: 
let $N^-_s=\prod_{ht(\beta)\ge s}U_{-\beta}$,
then $N^-_s$ is a normal subgroup of $N^-$. Denote by $N^{-,a}$ the product 
$N^{-,a}=\prod_{s\ge 1}N^-_s/N^-_{s+1}$,
then $N^{-,a}$ is a commutative unipotent group. We can identify $N^{-,a}$ naturally with the product
$\prod_{\beta\in R^+} {\mathbb G}_{a,\beta}$, viewed as a product of commuting additive groups.
Here ${\mathbb G}_{a,\beta}$ gets identified with the image of $U_{-\beta}$ in $N^-_{ht(\beta)}/N^-_{ht(\beta)+1}$.
The Lie algebra of $N^{-,a}$ is $\fn^{-,a}$.

The action of $B$ on $\fn^{-,a}$ can be lifted to an action on $N^{-,a}$ using the exponential map. 
To make this action more explicit, recall
that for two linearly independent roots $\alpha,\beta$ we know by Chevalley's commutator formula: there exist
complex numbers $c_{i,j,\alpha,\beta}$ such that
$$
x_\alpha(t) x_\beta(s) x^{-1}_\alpha(t)  x^{-1}_\beta(s)=\prod_{i,j>0} x_{i\alpha+j\beta}(c_{i,j,\alpha,\beta}t^i s^j)
$$
for all $s, t\in\bc$. The product is taken over all pairs $i, j \in\bz_{>0}$ such that
$i\alpha+j\beta$ is a root and in order of increasing height of the occurring roots.
We have for $m=\prod_{\beta\in R^+} x_{-\beta}(u_\beta)\in N^{-,a}$ and $x_\alpha(t)\in B$, $u_\beta,t\in\bc$:
\begin{equation}\label{chevalley}
x_\alpha(t)\circ m = \prod_{\beta\in R^+} x_{-\beta}(u_\beta+\sum_{\substack{i,j>0,\gamma\in R^+ \\ -\beta=i\alpha - j\gamma}} c_{i,j,\alpha,-\gamma}t^i u_\gamma^j).
\end{equation}
\begin{definition}
Denote by $G^a$  the semi-direct product $G^a\simeq B\semi N^{-,a}$, 
where $N^{-,a}$ is an abelian normal subgroup in $G^a$ and $B$
acts on $N^{-,a}$ via the action described above. 
\end{definition}

The subspaces 
$V(\la)_s=\U(\n^-)_s v_\la$ are stable with respect to the $B$- and the $\fb$-action, so we get an induced
action of $B$ as well as of $\fb$ on $V^a(\la)$. Since the application by an element $f\in\fn^-$ induces linear maps
$$
\begin{array}{rccc}
f:&V(\la)_s         &\rightarrow&V(\la)_{s+1}   \\
  &    \cup    &                   & \cup                 \\
  & V(\la)_{s-1} & \rightarrow&V(\la)_{s}  
\end{array}
$$
we get an induced endomorphism $\psi^a(f):V^a(\la)\rightarrow V^a(\la)$ with the property that
$\psi^a(f)\psi^a(f')-\psi^a(f')\psi^a(f):V^a(\la)\rightarrow V^a(\la)$ 
is the zero map for $f,f'\in\n^-$. 
Hence we get an induced representation of the abelian Lie algebra
$\fn^{-,a}$ and of its enveloping algebra $S(\n^{-,a})$, the symmetric algebra over $\fn^{-,a}$.
Note that $ V^a(\la)$ is a cyclic $S(\n^{-,a})$-module:
$$
V^a(\la)=S(\n^{-,a}).v_\la.
$$
The action of $\n^{-,a}$ on $V^a(\la)$ is compatible with the $B$-action on $V^a(\la)$ and on $\n^{-,a}$:
suppose $b\in B$, $f\in \n^-$ and $v\in V(\la)_s$, then 
$$
b(f.v)=(bfb^{-1})(bv)=(\ov{ad}(b)(f))bv+m.bv\ \text{for some $m\in\fb$},
$$
and hence $bf.v=(\ov{ad}(b)(f))bv$ in $V(\la)_{s+1}/V(\la)_s$. It follows:
\begin{prop}
$V^a(\la)$ is a $\fg^a$-module, it is a cyclic $S(\n^{-,a})$-module and a $B$-module.
The $B$-action on $S(\n^{-,a})$ is compatible with the $B$-action on $V^a(\la)=S(\n^{-,a}).v_\la$
\end{prop}

The action of $U_{-\beta}$ on $V(\la)$ is given by:
$$
\Psi(x_{-\beta}(t))(v)=\sum_{i\ge 0}t^i\psi\left(\frac{f^{i}_{\beta}}{i!}\right)(v)\text{\ for $v\in V(\la)$ and $t\in\bc$}
$$ 
and we get an induced action of $U_{-\beta}$ on $V^a(\la)$ by
$$
\Psi^a(x_{-\beta}(t))(v)=\sum_{i\ge 0}t^i\psi^a\left(\frac{f^{i}_{\beta}}{i!}\right)(v)\text{\ for $v\in V^a(\la)$ and $t\in\bc$.}
$$ 
The action of the various $U_{-\beta}$ on $V^a(\la)$ commutes and hence we get a representation 
$\Psi^a:N^{-,a}\rightarrow GL(V^a(\la))$. This action is compatible with the $B$-action on $V^a(\la)$
and hence:
\begin{prop}
$V^a(\la)$ is a representation space for $G^{a}$.
\end{prop}
In analogy to the classical construction we define:
\begin{definition}
The closure of the orbit $\ov{G^a.[v_\la]}\subseteq{\mathbb P}(V^a(\la))$ is called the degenerate flag variety $\Fl^a_\la$.
\end{definition}

\section{The Kostant lattice}
Let $G_\bz$ be a split and simple, simply connected algebraic group (see \cite{J}).
We assume without loss of generality $(G_\bz)_\bc=G$. We fix a split maximal torus
$T_\bz\subset G_\bz$ such that $T=(T_\bz)_\bc$ and a Borel subgroup $B_\bz\supset T_\bz$ such that
$B=(B_\bz)_\bc$. Let $\fg_\bz$, $\fb_\bz$, $\fn_\bz^+$ etc. be the Lie algebras, then we
have $\fg=\fg_\bz\otimes\bc$, $\fb=\fb_\bz\otimes\bc$ etc.  

Fix a Chevalley basis 
$$
\{f_\beta,e_\beta\,:\, \beta\in R^+; h_1,\ldots, h_n\}\subset \g_\bz,
$$
where $f_\beta\in \n^-_\bz$ (respectively $e_\beta\in \n^+_\bz$) is an element 
of the root space $\g_{-\beta,\bz}$ (respectively $\g_{\beta,\bz}$), and $h_i\in\h_\bz$. 

Let $\n^-_{\bz,s}=\sum_{ht(\beta)\ge s}\n^-_{\beta,\bz}$ be the Lie subalgebra formed by the root spaces
corresponding to roots of height at least $s$. The Lie subalgebra $\fn^-_{\bz,s+1}\subset \fn^-_{\bz,s}$
is an ideal, and the associated graded algebra $\n^{-,a}_\bz=\bigoplus_{s\ge 1} \fn^-_{\bz,s}/\fn^-_{\bz,s+1}$
is an abelian Lie algebra. We make $\n^{-,a}_\bz$ into a $B_\bz$- as well as a $\fb_\bz$-module by identifying 
the vector space $\fn^{-,a}_\bz$ with the quotient module $\g_\bz/\fb_\bz$, which is a $B_\bz$- respectively 
$\fb_\bz$-module via the adjoint action.

\begin{definition}
Denote by $\g^a_\bz$  the Lie algebra $\g^a_\bz= \fb_\bz\oplus \fn_\bz^{-,a}$, 
where $\fn_\bz^{-,a}$ is an abelian ideal in $\fg_\bz^a$ and $\fb_\bz$
acts on $\fn_\bz^{-,a}$ via the induced adjoint action described above. 
\end{definition}

We write
$e_{\beta}^{(m)}, f_{\beta}^{(m)}$ for the divided powers $\frac{f_{\beta}^{m}}{m!}$ and 
$\frac{e_{\beta}^{m}}{m!}$ in the enveloping algebra $U(\fg)$. We denote by $\left(\begin{array}{c}h_i\\ m\end{array}\right)$
the following element in $U(\fg)$:
$$
\left(\begin{array}{c}h_i\\ m\end{array}\right)=\frac{h_i(h_i-1)\cdots(h_i-m+1)}{m!}.
$$
Let now $U_\bz(\g)$ be the Kostant lattice in $U(\g)$, i.e. the subalgebra generated by the 
$\left(\begin{array}{c}h_i\\ m\end{array}\right)$ and the divided powers $e_{\beta}^{(m)}, f_{\beta}^{(m)}$.
We identify $U_\bz(\g)$ with
$\text{Dist} (G_\bz)$, the algebra of distributions or the hyperalgebra of $G_\bz$.
We fix an enumeration of the positive roots $\{\beta_1,\ldots,\beta_N\}$. 
Given an $N$-tuple ${\mathbf m}=(m_1,\ldots,m_N)$ of non-negative integers, we set
$$
f^{(\mathbf m)}= f_{\beta_1}^{(m_1)}\cdots f_{\beta_N}^{(m_N)},
e^{(\mathbf m)}= e_{\beta_1}^{(m_1)}\cdots e_{\beta_N}^{(m_N)},
$$
and given an $n$-tuple ${\mathbf \ell}=(\ell_1,\ldots,\ell_n)$, set
$$
h^{(\mathbf \ell)}=\left(\begin{array}{c}h_1\\ \ell_1\end{array}\right)\cdots \left(\begin{array}{c}h_n\\ \ell_n\end{array}\right).
$$
 The ordered monomials
$$
f^{(\mathbf m)} h^{(\mathbf \ell)} e^{(\mathbf k)}, \ \text{where ${\mathbf m},{\mathbf k}$ are $N$-tuples, ${\mathbf \ell}$ is an $n$-tuple of natural numbers,}
$$
form a $\bz$-basis of $U_\bz(\g)$ as a free $\bz$-module. The subalgebras
$U_\bz(\n^-)$ and $U_\bz(\n^+)$  admit the ordered monomials
$$
\left\{ f^{(\mathbf m)}\mid m_1,\ldots,m_N\in \bz_{\ge 0}\right\}
$$
respectively
$$
\left\{ e^{(\mathbf m)}\mid m_1,\ldots,m_N\in \bz_{\ge 0}\right\}
$$
as bases. 

Let $\U_\bz (\n^-)_s$ be the $\bz$-span of the monomials
of degree at most $s$:
\begin{equation}\label{dfz}
U_\bz (\n^-)_s=\langle f_{\gamma_1}^{(m_1)}\dots f_{\gamma_\ell}^{(m_\ell)} \mid m_1+\ldots+m_\ell\le s, \gamma_1,\ldots,\gamma_\ell \in R^+ \rangle_\bz,
\end{equation}
where the degree of $f_{\gamma_1}^{(m_1)}\dots f_{\gamma_\ell}^{(m_\ell)}$ is the sum $m_1+\ldots+m_\ell$.
Since changing the ordering is commutative up to terms of smaller degree, the $\U_\bz (\n^-)_s$ define a filtration of the algebra $\U_\bz (\n^-)$.
By abuse of notation denote by $S_\bz(\n^{-,a})$ the associated graded algebra. Note that $\fn_\bz^{-,a}\subset S_\bz(\n^{-,a})$. In fact,
$S_\bz(\n^{-,a})$ is a divided power analogue of the symmetric algebra
over $\fn_\bz^{-,a}$. This algebra can be described as the quotient of
a polynomial algebra in infinitely many generators (the ``symbols'' $\f_\beta^{(m)}$): $\bz[\f_\beta^{(m)}\mid m\in\bz_{\ge 0},\beta\in R^+]$
modulo the ideal ${\mathfrak J}$ generated by the following identities:
\begin{equation}\label{symbolrelation}
{\mathfrak J}=\langle \f_\beta^{(m)}\f_\beta^{(k)}-
\left(\begin{array}{c}
m+k\\ m
\end{array}\right)\f_\beta^{(m+k)}\mid k,m \ge 1,\beta\in R^+\rangle.
\end{equation}
So we have:
$$
S_\bz(\n^{-,a})\simeq\bz[\f_\beta^{(m)}\mid m\in\bz_{\ge 0},\beta\in R^+]/{\mathfrak J}.
$$
The isomorphism above sends the basis given by classes of the monomials in the 
symbols $\f_{\beta_1}^{(m_1)}\cdots \f_{\beta_N}^{(m_N)}$ to the basis of $S_\bz(\n^{-,a})$ given by
the monomials $f_{\beta_1}^{(m_1)}\cdots f_{\beta_N}^{(m_N)}$. 

Let $U^+_\bz(\h+\n^+)\subset U_\bz(\g)$ be the span of the monomials $h^{(\mathbf \ell)} e^{(\mathbf k)}$ such that 
$\sum_{i=1}^n \ell_i+\sum_{j=1}^N k_j>0$. The natural map which sends a monomial to its class in the quotient:
$$
U_\bz(\n^-)\rightarrow U_\bz(\g)/U_\bz(\n^-)U^+_\bz(\h+\n^+), \quad f^{({\mathbf m})}\rightarrow \overline{f^{({\mathbf m})}},
$$
is an isomorphism of free $\bz$-modules. Recall that $U_\bz(\g)$ is naturally a $B_\bz$-module and a 
$U_\bz (\fb)$-module via the adjoint action,
and $U_\bz(\n^-)U^+_\bz(\h+\n^+)$ is a proper submodule. Via the identification above, we get an induced
structure on $U_\bz(\n^-)$ as a $B_\bz$-module and a $U_\bz(\fb)$-module. The filtration of $\U_\bz (\n^-)$ by the $\U_\bz (\n^-)_s$ is
stable under this $B_\bz$- and $\U_\bz (\fb)$-action and hence:
\begin{lem}\label{plusonsym}
The $B_\bz$-module structure and the $U_\bz(\fb)$-module structure on $\U_\bz (\n^-)$ induce
a $B_\bz$-module structure and a $U_\bz(\fb)$-module structure on $S_\bz(\n^{-,a})$.
\end{lem}

For a dominant integral weight  $\la=m_1\omega_1 + \dots + m_n\omega_n$ fix a highest weight vector $v_\la$ and let
$V_\bz(\la)=U_\bz(\g)v_\la\subset V(\la)$ be the corresponding lattice in the complex representation space.
Since $V_\bz(\la)=\U_\bz(\n^-)v_\la$, the filtration \eqref{dfz} induces an increasing  filtration $V_\bz(\la)_s$ on $V_\bz(\la)$:
\[
V_\bz(\la)_s=\U_\bz(\n^-)_s v_\la.
\]
We denote the associated graded space by $ V^a_\bz(\la)$. Since $B_\bz V_\bz(\la)_s\subset V_\bz(\la)_s$,
$V^a_\bz(\la)$ becomes naturally a $B_\bz$-module.
The application by an element $f_\beta^{(m)}\in U_\bz(\fn^-)$ provides linear maps for all $s$:
$$
\begin{array}{rccc}
f_\beta^{(m)}:&V_\bz(\la)_s         &\rightarrow&V_\bz(\la)_{s+m}   \\
  &    \cup    &                   & \cup                 \\
  & V_\bz(\la)_{s-1} & \rightarrow&V_\bz(\la)_{s+m-1},  
\end{array}
$$
and we get an induced endomorphism $\psi^a(f_\beta^{(m)}):V_\bz^a(\la)\rightarrow V_\bz^a(\la)$ such that
$\psi^a(f_\beta^{(m)})\psi^a(f_\gamma^{(\ell)})= \psi^a(f_\gamma^{(\ell)})\psi^a(f_\beta^{(m)})$,
and hence we get an induced representation of the abelian Lie algebra
$\fn_\bz^{-,a}$ and of the algebra $S_\bz(\n^{-,a})$.
Note that $V_\bz^a(\la)$ is a cyclic $S_\bz(\n^{-,a})$-module:
\[
 V^a_\bz(\la)=S_\bz(\n^{-,a})v_\la.
\]
The action of $S_\bz(\n^{-,a})$ on $V_\bz^a(\la)$ is compatible with the 
$B_\bz$-action on $S_\bz(\n^{-,a})$ and on $ V^a(\la)$, so summarizing we have:
\begin{prop}
$V_\bz^a(\la)$ is a $\fg_\bz^a$-module, it is a cyclic $S_\bz(\n^{-,a})$-module and a $B_\bz$-module.
The $B_\bz$-action on $S_\bz(\n^{-,a})$ is compatible with the $B_\bz$-action on $V_\bz^a(\la)=S_\bz(\n^{-,a}).v_\la$.
\end{prop}
For a positive root $\beta$ let $U_{-\beta,\bz}\subset G_\bz$ be the closed root subgroup corresponding 
to the root $-\beta$. We 
denote by $x_{-\beta}:{\mathbb G}_{a,\bz,\beta}\rightarrow U_{-\beta,\bz}$ a fixed isomorphism of the root subgroup
with the additive group ${\mathbb G}_{a,\bz}$. We add the root as an index to indicate that this copy ${\mathbb G}_{a,\bz,\beta}$
of the additive group is supposed to be identified with $U_{-\beta,\bz}$.

As in the case before over the complex numbers, the group $N^-_\bz$ admits a filtration by a sequence of normal subgroups: 
set $N^-_{\bz,s}=\prod_{ht(\beta)\ge s}U_{-\beta,\bz}$, the product $N_\bz^{-,a}=\prod_{s\ge 1}N^-_{\bz,s}/N^-_{\bz,s+1}$,
is a commutative group. We can identify $N_{\bz}^{-,a}$ naturally with the product
$\prod_{\beta\in R^+} {\mathbb G}_{a,\bz,\beta}$, viewed as a product of commuting additive groups.
Again, ${\mathbb G}_{a,\bz,\beta}$ gets identified with the image of $U_{-\beta,\bz}$ in $N^-_{\bz,ht(\beta)}/N^-_{\bz,ht(\beta)+1}$.
The Lie algebra of $N_{\bz}^{-,a}$ is $\fn_\bz^{-,a}$.

The action of $U_{-\beta,\bz}$ on $V_\bz(\la)$ is given by:
$$
\Psi(u_{-\beta}(t))(v)=\sum_{i\ge 0}t^i\psi(f^{(i)}_{\beta})(v)\text{\ for $v\in V_\bz(\la)$ and $t\in\bz$}
$$ 
and we get an induced action of $U_{-\beta,\bz}$ on $V^a_\bz(\la)$ by
$$
\Psi^a(u_{-\beta}(t))(v)=\sum_{i\ge 0}t^i\psi^a(f^{(i)}_{\beta})(v)\text{\ for $v\in V^a_\bz(\la)$ and $t\in\bz$}.
$$ 
The action of the various $U_{-\beta,\bz}$ on $V_\bz^a(\la)$ commute and hence we get a representation 
$\Psi^a:N_\bz^{-,a}\rightarrow GL(V_\bz^a(\la))$. Since we started with a Chevalley basis, 
by \cite{Ste68}, \S 6, or \cite{Tit87},\S 3.6, the coefficients
in \eqref{chevalley} are integral, so we get an action of $B_\bz$ on $N_\bz^{-,a}$. Denote by
$G^a_{\bz}$ the semi-direct product $B_\bz\semi N_\bz^{-,a}$. The actions of $B_\bz$ and $N_\bz^{-,a}$
on $V_\bz^a(\la)$ are compatible and hence we get 
\begin{prop}
$V_\bz^a(\la)$ is a $G_\bz^a$-module.
\end{prop}
As a consequence, given a field $k$, we have the group $G^a_k=(G^a_\bz)_k$, the representation
space $V^a_k=(V_\bz^a)_k$ and the degenerate flag variety $\Fl_{\la,k}^a:=\ov{G^a_k.[v_\la]}\subset {\mathbb P}(V_k^a(\la))$.
Here are some natural questions: 
\begin{itemize}
\item is the graded character of $V_k^a(\la))$ independent of the characteristic?
\item is $V_\bz^a(\la)$ torsion free?
\end{itemize}
An explicit monomial basis for $V_\bc^a(\la)$ has been constructed  for $G=SL_n$ in \cite{FFoL1}
and for $G=Sp_{2n}$ in \cite{FFoL2}. Another natural question:
\begin{itemize}
\item is this basis of $V^a(\la)$ compatible with the lattice construction in this section? Or, to put it differently:
is $V_\bz^a(\la)$ a free $\bz$-module with basis $\{f^{(\bs)} v_\lambda\mid \bs\in S(\la)\}$? (For the notation see the next sections.)
\end{itemize}
The aim of the next sections is to give an affirmative answer to these questions for $G=SL_n$ and 
$G=Sp_{2n}$.
\section{Roots and relations in type ${\tt A}$ and ${\tt C}$}
Let $R^+$ be the set of positive roots of $\msl_{n+1}$.  Let
$\al_i$, $\omega_i$ $i=1,\dots,n$ be the simple roots and the fundamental weights.
All roots of $\msl_{n+1}$ are of the form $\al_p + \al_{p+1} +\dots + \al_q$
for some $1\le p\le q\le n$. In the following we denote such a
root by $\al_{p,q}$, for example $\al_i=\al_{i,i}$.

Let now $R^+$ be the set of positive roots of $\msp_{2n}$.  
Let
$\al_i$, $\omega_i$ $i=1,\dots,n$ be the simple roots and the fundamental weights.
All positive roots of $\msp_{2n}$ 
can be divided into two groups:
\begin{gather*}
\al_{i,j}=\al_i+\al_{i+1}+\dots +\al_j,\ 1\le i\le j\le n,\\
\alpha_{i, \ol{j}} = \alpha_i + \alpha_{i+1} + \ldots + 
\alpha_n + \alpha_{n-1} + \ldots + \alpha_j, \ 1\le i\le j\le n
\end{gather*}
(note that $\al_{i,n}=\al_{i,\ol n}$).
We will use the following short versions
$$\al_i = \al_{i,i},\ \al_{\ol{i}} = \al_{i,\ol{i}}.$$
We recall the usual order on the alphabet $J = \{1, \ldots, n, \ol{n-1}, \ldots, \ol{1}\}$
\begin{equation} \label{sa}
1 <2 < \ldots < n-1 < n < \ol{n-1} < \ldots < \ol{1}.
\end{equation}

Let $\fg=\n^+\oplus\h\oplus \n^-$ be the Cartan decomposition.
By Lemma~\ref{plusonsym}, the $U_\bz(\n^+)$-module structure on $\U_\bz (\n^-)$ induces
a $U_\bz(\n^+)$-module structure on $S_\bz(\n^{-,a})$.
We want to make this action more explicit for $\fg$ of type ${\tt A}$ and ${\tt C}$. 

If $\al=\beta$ or if the root vectors commute, then 
\begin{equation} \label{ad0}
(\text{ad}\, e^{(k)}_\al)(f^{(m)}_\beta)=0.
\end{equation}
If $\alpha,\gamma,\beta=\alpha+\gamma$ are positive roots
spanning a subsystem of type ${\tt A}_2$, then 
\begin{equation} \label{ad1}
(\text{ad}\, e^{(k)}_\al)(f^{(m)}_\beta) =
\begin{cases}
\pm f_{\gamma}^{(k)}f^{(m-k)}_\beta,\  \text{ if } k\le m,\\
0,\ \text{ otherwise}.
\end{cases}
\end{equation}
If $\alpha,\gamma,\alpha+\gamma,\alpha+2\gamma$ span a subrootsystem of
type ${\tt B}_2={\tt C}_2$, then 
\begin{equation} \label{ad2}
(\text{ad}\, e^{(k)}_\al)(f^{(m)}_{\alpha+\gamma}) =
\begin{cases}
\pm f_{\gamma}^{(k)}f^{(m-k)}_{\alpha+\gamma},\  \text{ if } k\le m,\\
0,\ \text{ otherwise},
\end{cases},
\end{equation}
and 
\begin{equation} \label{ad3}
(\text{ad}\, e^{(k)}_{\al+\gamma})(f^{(m)}_{\alpha+2\gamma}) =
\begin{cases}
\pm f_{\gamma}^{(k)}f^{(m-k)}_{\alpha+2\gamma},\  \text{ if } k\le m,\\
0,\ \text{ otherwise},
\end{cases},
\end{equation}
and 
\begin{equation} \label{ad4}
(\text{ad}\, e^{(k)}_\gamma)(f^{(m)}_{\alpha+\gamma}) =
\begin{cases}
\pm 2^k f_{\alpha}^{(k)}f^{(m-k)}_{\alpha+\gamma},\  \text{ if } k\le m,\\
0,\ \text{ otherwise},
\end{cases}
\end{equation}
and 
\begin{equation} \label{ad5}
(\text{ad}\, e^{(k)}_\gamma)(f^{(m)}_{\alpha+2\gamma}) =
\begin{cases}
\pm  f_{\alpha+\gamma}^{(k)}f^{(m-k)}_{\alpha+2\gamma}\\
\hskip 15pt+\sum_{\stackrel{c>m-k}{a+b+c=m}}r_{a,b,c} f_{\alpha}^{(a)}f_{\alpha+\gamma}^{(b)}f^{(c)}_{\alpha+2\gamma},\  \text{ if } k\le m,\\
0,\ \text{ otherwise},
\end{cases}
\end{equation}
where the coefficients $r_{a,b,c} $ are integers.

\section{The spanning property for $SL_{n+1}$}\label{spansl}
We first recall the definition of a Dyck path in the $SL_{n+1}$-case:
\begin{dfn}\rm
A {\it Dyck path} (or simply {\it a path}) is a sequence
\[
\bp=(\beta(0), \beta(1),\dots, \beta(k)), \ k\ge 0
\]
of positive roots satisfying the following conditions:
\begin{itemize}
\item[{\it a)}] the first and last elements are simple roots. More precisely,
$\beta(0)=\al_i$ and $\beta(k)=\al_j$ for some $1\le i\le j\le n$;
\item[{\it b)}] the elements in between obey the following recursion rule:
If $\beta(s)=\al_{p,q}$ then the next element in the sequence is of the form either
$\beta(s+1)=\al_{p,q+1}$  or $\beta(s+1)=\al_{p+1,q}.$
\end{itemize}
\end{dfn}

\begin{exam}
Here is an example for a Dyck path for $\msl_6$:
\[
\bp =(\al_2,\al_2+\al_3,\al_2+\al_3+\al_4,\al_3+\al_4,\al_4,\al_4+\al_5,\al_5).
\]
\end{exam}
\vskip 5pt\noindent
For a multi-exponent $\bs=\{s_ \beta\}_{\beta>0}$, $s_ \beta\in\Z_{\ge 0}$, let $f^{(\bs)}$ be the element
\[
f^{(\bs)}=\prod_{\beta\in R^+} f_ \beta^{(s_ \beta)}\in S_\bz(\fn^{-,a}).
\]
\begin{dfn}
For an integral dominant $\msl_{n+1}$-weight $\la=\sum_{i=1}^n m_i\omega_i$
let $S(\la)$ be the set of all multi-exponents $\bs=(s_ \beta)_{\beta\in R^+}\in\bz_{\ge 0}^{R^+}$
such that for all Dyck paths $\bp =(\beta(0),\dots, \beta(k))$
\begin{equation}\label{upperbound}
s_{\beta(0)}+s_{\beta(1)}+\dots + s_{\beta(k)}\le m_i+m_{i+1}+\dots + m_j,
\end{equation}
where $\beta(0)=\al_i$ and $\beta(k)=\al_j$.
\end{dfn}

The space $ V^a_\bz(\la)$ is endowed with the structure of a cyclic $S_\bz(\n^{-,a})$-module,
 hence  $ V^a_\bz(\la)=S_\bz(\n^{-,a})/I_\bz(\la)$ for some ideal  $I_\bz(\la)\subseteq S_\bz(\n^{-,a})$.
Our goal is to prove that the elements
$f^{(\bs)} v_\la$, $\bs\in S(\la)$, span $ V^a_\bz(\la)$.

Let $\la=m_1\omega_1 + \dots + m_n\omega_n$. The strategy is as follows:
$f_\al^{((\la,\al)+1)}v_\la=0$ in $V_\bz(\la)$ for all positive roots $\al$,
so for $\alpha=\al_i + \dots + \al_j$, $i\le j$, we have the relation
\[
f_{\al_i + \dots + \al_j}^{(m_i+\dots + m_j+1)} \in I_\bz(\la).
\]
In addition we have the operators $e_\al^{(m)}$ 
acting on $ V^a_\bz(\la)$. 
We note that $I_\bz(\la)$ is stable with respect to the induced action of the $e_\al^{(m)}$ 
on $S_\bz(\n^{-,a})$
(Lemma~\ref{plusonsym}). 
By applying the operators $e_\al^{(m)}$ to $f_{\al_i + \dots + \al_j}^{(m_i+\dots + m_j+1)}$,
we obtain new relations. We prove that these relations are enough to
rewrite any vector $f^{(\bt)} v_\lam$ as an integral linear combination of $f^{(\bs)} v_\lam$ with
$\bs\in S(\la)$.

By the degree $\deg \bs$ of a multi-exponent we mean the degree of the corresponding
monomial in $S_\bz(\fn^{-,a})$, i.e. $\deg \bs=\sum s_{i,j}$.

We are going to define an {\it order} on the monomials in $S_\bz(\fn^{-,a})$. To begin with, we define a total
order on the $ f_{i,j}$, $1\le i\le j\le n$. We say that $(i,j)\succ (k,l)$ if $i>k$
or if $i=k$ and $j>l$. Correspondingly we say that $f_{i,j}\succ f_{k,l}$ if $(i,j)\succ (k,l)$, so
\[
f_{n,n}\succ f_{n-1,n}\succ f_{n-1,n-1}\succ f_{n-2,n}\succ\ldots \succ f_{2,3}\succ f_{2,2}\succ f_{1,n}\succ\ldots\succ f_{1,1}.
\]
We use a sort of associated {\it homogeneous lexicographic ordering} on the set of multi-exponents, i.e. 
for two multi-exponents $\bs$ and $\bt$ we write $\bs\succ\bt$:
\begin{itemize}
\item if $\deg \bs>\deg \bt$,
\item if $\deg \bs = \deg \bt$ and there exist $1\le i_0\le j_0\le n$ such
that $s_{i_0j_0}>t_{i_0j_0}$ and for $i>i_0$ and ($i=i_0$ and $j>j_0$) we have $s_{i,j}=t_{i,j}$.
\end{itemize}
We use the ``same" total order on the set of monomials, i.e. $f^{(\bs)}\succ f^{(\bt)}$ if and only if $\bs\succ \bt$.

\begin{prop}\label{straightening}
Let $\bp=(p(0),\dots,p(k))$ be a Dyck path with $p(0)=\al_i$ and $p(k)=\al_j$.
Let $\bs$ be a multi-exponent supported on $\bp$, i.e. $s_\al=0$ for $\al\notin\bp$.
Assume further that
\[
\sum_{l=0}^k s_{p(l)}>m_i+\dots +m_j.
\]
Then there exist some constants $c_{\bt}\in\bz $ labeled by multi-exponents $\bt$ such that
\begin{equation}\label{straighteninglawsl}
f^{(\bs)} +\sum_{\bt\prec\bs} c_\bt f^{(\bt)} \in I_\bz(\lam)
\end{equation}
($\bt$ does not have to be supported on $\bp$).
\end{prop}
\begin{rem}\rm
We refer to \eqref{straighteninglawsl} as a {\it straightening law} because it implies
$$
f^{(\bs)} = - \sum_{\bt\prec \bs} c_\bt f^{(\bt)} \text{\ in\ }S_\bz(\fn^{-,a})/I_\bz(\lam)\simeq  V^a_\bz(\lam).
$$
\end{rem}
\begin{proof}
We start with the case $p(0)=\al_1$ and $p(k)=\al_n$ (so, $k=2n-2$).
This assumption is just for convenience. In the general case 
$\bp$ starts with $p(0)=\al_i$, $p(k)=\al_j$ and one would start with the relation
$f_{i,j}^{(m_i+\dots +m_j+1)} \in I_\bz(\lam)$ instead of the relation
$f_{1,n}^{(m_1+\dots +m_n+1)} \in I_\bz(\lam)$ below.

So from now on we assume without loss of generality that $p(0)=\al_1$ and $p(k)=\al_n$.
%
In the following we use the differential operators $\pa^{(k)}_\al$ defined by
\begin{equation} \label{pa1}
\pa^{(k)}_\al f^{(m)}_\beta=
\begin{cases}
f_{\beta-\al}^{(k)}f^{(m-k)}_\beta,\  \text{ if }  \beta-\al\in\triangle^+\ \text{ and }\ k\le m,\\
0,\ \text{ otherwise}.
\end{cases}
\end{equation}
The operators $\pa^{(k)}_\al$ satisfy the property
\[
\pa^{(k)}_\al f^{(m)}_\beta = \pm(\text{ad}\, e^{(k)}_\al)(f^{(m)}_\beta).
\]
In the following we use very often the following consequence:
if a monomial $f^{(m_1)}_{\beta_1}\dots f^{(m_l)}_{\beta_l}\in I_\bz(\la)$, then for any sequence of positive roots $\al_1,\dots,\al_s$
and any sequence of integers $k_1,\ldots,k_s\in\bz_{>0}$ we have:
\[
\pa^{(k_1)}_{\al_1}\dots\pa^{(k_s)}_{\al_s} f^{(m_1)}_{\beta_1}\dots f^{(m_l)}_{\beta_l}\in I_\bz(\la).
\]
Since $f_{1,n}^{(m_1+\dots +m_n+1)} v_\lam=0$
in $ V^a_\bz(\la)$ and $s_{p(0)}+\dots + s_{p(k)}>m_1+\dots +m_n$ by assumption, it follows that
\[
f_{1,n}^{(s_{p(0)}+\dots + s_{p(k)})} \in I(\lam).
\]
Write $\pa^{(m)}_{i,j}$ for $\pa^{(m)}_{\al_{i,j}}$, and for $i,j=1,\dots,n$ set
\[
s_{\bullet, j}=\sum_{i=1}^j s_{i,j},\quad s_{i,\bullet}=\sum_{j=i}^n s_{i,j}.
\]
We first consider the vector
\begin{equation}\label{pape1}
\pa_{n,n}^{(s_{\bullet, n-1})}\pa_{n-1,n}^{(s_{\bullet, n-2})}\dots \pa_{2,n}^{(s_{\bullet, 1})}
f_{1,n}^{(s_{p(0)}+\dots + s_{p(k)})} \in I_\bz(\lam).
\end{equation}
By means of formula \eqref{pa1} we get:
\[
\pa_{2n}^{(s_{\bullet, 1})}
f_{1,n}^{(s_{p(0)}+\dots + s_{p(k)})}= f_{1,n}^{(s_{p(0)}+\dots + s_{p(k)}-s_{\bullet, 1})}f_{1,1}^{(s_{\bullet, 1})}
\]
and
\[
\pa_{3n}^{(s_{\bullet, 2})}\pa_{2n}^{(s_{\bullet, 1})}
f_{1,n}^{(s_{p(0)}+\dots + s_{p(k)})} =f_{1,n}^{(s_{p(0)}+\dots + s_{p(k)}-s_{\bullet 1}-s_{\bullet 2})}f_{1,1}^{(s_{\bullet 1})}f_{1,2}^{(s_{\bullet 2})}.
\]
Summarizing, the vector \eqref{pape1} is equal to
\[
f_{1,1}^{(s_{\bullet, 1})}f_{1,2}^{(s_{\bullet ,2})}\dots f_{1,n}^{(s_{\bullet, n})} \in I_\bz(\lam).
\]
To prove the proposition, we apply more differential operators to the monomial
$f_{1,1}^{(s_{\bullet ,1})}f_{1,2}^{(s_{\bullet, 2})}\dots f_{1,n}^{(s_{\bullet, n})}$. Consider the following element in $I_{\bz}(\lam)\subset S_\bz(\fn^{-,a})$:
\begin{equation}\label{Aoperator}
A=\pa_{1,1}^{(s_{2,\bullet})}\pa_{1,2}^{(s_{3,\bullet})}\dots \pa_{1,n-1}^{(s_{n,\bullet})}
f_{1,1}^{(s_{\bullet ,1})}f_{1,2}^{(s_{\bullet, 2})}\dots f_{1,n}^{(s_{\bullet,n})}.
\end{equation}
{\bf We claim:}
\begin{equation}\label{Aoperatorequation}
A =\sum_{\bt \preceq \bs} c_\bt f^{(\bt)} \text{\ where $c_\bs =1 $}.
\end{equation}
Now $A \in I_\bz(\lam)$ by construction, so the claim proves the proposition.
\vskip 5pt\noindent
{\it Proof of the claim:}
In order to prove the claim we need to introduce some more notation. For $j=1,\ldots,n-1$ set
\begin{equation}\label{Ajoperator}
A_j=\pa_{1,j}^{(s_{j+1,\bullet})}\pa_{1,j+1}^{(s_{j+2,\bullet})}\dots \pa_{1,n-1}^{(s_{n,\bullet})}
f_{1,1}^{(s_{\bullet, 1})}f_{1,2}^{(s_{\bullet, 2})}\dots f_{1,n}^{(s_{\bullet, n})},
\end{equation}
so $A_1=A$. To start an inductive procedure, we begin with $A_{n-1}$:
$$
A_{n-1}=\pa_{1,n-1}^{(s_{n,\bullet})} f_{1,1}^{(s_{\bullet, 1})}f_{1,2}^{(s_{\bullet, 2})}\dots f_{1,n}^{(s_{\bullet, n})}.
$$
Now $s_{n,\bullet}=s_{n,n}$ and $\pa_{1,n-1}^{(x)}f^{(y)}_{1,j}=0$ for $j\not=n$, so
\begin{equation}\label{Aoperatorequation1}
A_{n-1}= f_{1,1}^{(s_{\bullet, 1})}f_{1,2}^{(s_{\bullet, 2})}\dots f_{1,n}^{(s_{\bullet, n}-s_{n,n})}f_{n,n}^{(s_{n,n})}.
\end{equation}
We proceed with the proof using decreasing induction. Since the induction procedure
is quite involved and the initial step does not reflect the problems
occurring in the procedure, we discuss for convenience the case $A_{n-2}$
separately.

Consider $A_{n-2}$, we have:
$$
A_{n-2}=\pa_{1,n-2}^{(s_{n-1,\bullet})}
f_{1,1}^{(s_{\bullet, 1})}f_{1,2}^{(s_{\bullet, 2})}\dots f_{1,n}^{(s_{\bullet, n}-s_{n,n})}f_{n,n}^{(s_{n,n})}.
$$
Now $\pa^{(k)}_{1,n-2}f^{(m)}_{1,j}=0$ for $j\not=n-1,n$, $\pa^{(k)}_{1,n-2}f^{(m)}_{n,n}=0$, and 
$\pa^{(k)}(xy)
=\sum_{i=0}^k \pa^{(k-i)}(x)\pa^{(i)}(y)$, so
$$
\begin{array}{rcl}
A_{n-2}&=&\sum_{\ell=0}^{s_{n-1 ,\bullet}}
f_{1,1}^{(s_{\bullet, 1})}f_{1,2}^{(s_{\bullet, 2})}\dots \\
&&\quad \dots f_{1,n-1}^{(s_{\bullet, n-1}-s_{n-1,\bullet}+\ell)}f_{1,n}^{(s_{\bullet, n}-s_{n,n}-\ell)}
f_{n-1,n-1}^{(s_{n-1,\bullet}-\ell)}f_{n-1,n}^{(\ell)}f_{n,n}^{(s_{n,n})}.
\end{array}
$$
We need to control which divided powers $f_{n-1,n}^{(\ell)}$ can occur. Recall that $\bs$ has support in $\bp$.
If $\al_{n-1}\not\in \bp$, then $s_{n-1,n-1}=0$ and
$s_{n-1,\bullet}=s_{n-1,n}$, so $f_{n-1,n}^{(s_{n-1,n})}$ is the highest divided power occurring in the sum.
Next suppose $\al_{n-1}\in \bp$. This implies $\al_{j,n}\not \in \bp$ unless $j=n-1$ or $n$. Since
$\bs$ has support in $\bp$, this implies
$$
s_{\bullet, n}=s_{1,n}+\ldots + s_{n-1,n}+s_{n,n}=s_{n-1,n}+s_{n,n},
$$
and hence again the highest divided power of $f_{n-1,n}$ which can occur is $f_{n-1,n}^{(s_{n-1,n})}$,
and the coefficient is $1$. So we can write
\begin{equation}\label{Aoperatorequation2}
A_{n-2}=\sum_{\ell=0}^{s_{n-1,n}} 
f_{1,1}^{(s_{\bullet, 1})}
\dots f_{1,n-1}^{(s_{\bullet, n-1}-s_{n-1,\bullet}+\ell)}f_{1,n}^{(s_{\bullet, n}-s_{n,n}-\ell)}
f_{n-1,n-1}^{(s_{n-1,\bullet}-\ell)}f_{n-1,n}^{(\ell)}f_{n,n}^{(s_{n,n})}.
\end{equation}
For the inductive procedure we make the following assumption:

$A_{j}$ is a sum with integral coefficients of monomials of the form
\begin{equation}\label{Aoperatorequation3}
\underbrace{f_{1,1}^{(s_{\bullet, 1})}\ldots f_{1,j}^{(s_{\bullet, j})}f_{1,j+1}^{(s_{\bullet, j+1}-*)}\ldots f_{1,n}^{(s_{\bullet, n}-*)}}_X
\underbrace{f_{j+1,j+1}^{(t_{j+1,j+1})}f_{j+1,j+2}^{(t_{j+1,j+2})}\dots f_{n-1,n}^{(t_{n-1,n})}f_{n,n}^{(t_{n,n})}}_Y
\end{equation}
having the following properties:
\begin{itemize}
\item[{\it i)}] With respect to the homogeneous lexicographic ordering, all the multi-exponents of the summands, except one,
are strictly smaller than $\bs$.
\item[{\it ii)}] More precisely, there exists a pair $(k_0,\ell_0)$ such that $k_0\ge j+1$,
$s_{k_0 \ell_0}>t_{k_0 \ell_0}$ and $s_{k\ell}=t_{k\ell}$ for all $k>k_0$ and all pairs $(k_0.\ell)$ such that $\ell>\ell_0$.
\item[{\it iii)}] The only exception is the summand such that $t_{\ell,m}=s_{\ell,m}$ for all $\ell\ge j+1$
and all $m$, and in this case the coefficient is equal to $1$.
\end{itemize}
The calculations above show that this assumption holds for $A_{n-1}$ and $A_{n-2}$.

We start now with the induction procedure and we consider $A_{j-1}= \pa_{1,j-1}^{(s_{j,\bullet})}A_{j}$.
Note that $\pa^{(k)}_{1,j-1} f^{(m)}_{1,\ell}=0$ for $\ell < j$, and for $\ell\ge j$ we have 
$\pa^{(p)}_{1,j-1} f^{(q)}_{1,\ell}=f^{(p)}_{j,\ell}f^{(q-p)}_{1,\ell}$ for $p\le q$, and the result is $0$
for $p> q$.

Furthermore, $\pa^{(p)}_{1,j-1} f^{(q)}_{k,\ell}=0$ for $k\ge j+1$, so applying $\pa^{(p)}_{1,j-1}$ to a summand of the 
form (\ref{Aoperatorequation3}) does not change the $Y$-part in (\ref{Aoperatorequation3}). Summarizing,
applying $\pa_{1,j-1}^{(s_{j,\bullet})}$ to a summand of the form (\ref{Aoperatorequation3}) gives a sum of
monomials of the form
\begin{equation}\label{Aoperatorequation4}
\begin{array}{rl}
\underbrace{f_{1,1}^{(s_{\bullet ,1})}\ldots f_{1,j-1}^{(s_{\bullet, j-1})}f_{1,j}^{(s_{\bullet,j}-*)}\ldots f_{1,n}^{(s_{\bullet,n}-*)}}_{X'}&
\underbrace{f_{j,j}^{(t_{j,j})}\ldots f_{j,n}^{(t_{j,n})} }_Z\\
&\hskip -50pt \underbrace{f_{j+1,j+1}^{(t_{j+1,j+1})}f_{j+1,j+2}^{(t_{j+1,j+2})}\dots
f_{n,n}^{(t_{n,n})}}_Y.
\end{array}
\end{equation}
We have to show that these summands satisfy again the conditions ${\it i)}$--${\it iii)}$ above (but now for the index $(j-1)$). If we start
in (\ref{Aoperatorequation3}) with a summand which is not the maximal summand, but such that {\it i)} and {\it ii)} hold for the index $j$,
then the same holds obviously also for the index $(j-1)$ for all summands in (\ref{Aoperatorequation4}) because the $Y$-part
remains unchanged.

So it remains to investigate the summands of the form (\ref{Aoperatorequation4}) obtained by applying $\pa_{1j-1}^{(s_{j,\bullet})}$ to the only
summand in (\ref{Aoperatorequation3}) satisfying {\it iii)}.

To formalize the arguments used in the calculation for $A_{n-2}$ we need the following notation.
Let $1\le k_1\le k_2\le\dots\le k_n\le n$ be numbers defined by
\[
k_i=\max\{j:\ \al_{i,j}\in\bp\}.
\]
For convenience we set $k_0=1$.
\begin{exam}
For $\bp=(\al_{11},\al_{12},\dots,\al_{1n},\al_{2n},\dots,\al_{n,n})$ we have
$k_i=n$ for all  $i=1,\ldots,n$.
\end{exam}
%
Since $\bs$ is supported on $\bp$ we have
\begin{equation}\label{sumpath}
s_{i,\bullet}=\sum_{\ell=k_{i-1}}^{k_i} s_{i,\ell},\ s_{\bullet, \ell}=\sum_{i:\ k_{i-1}\le \ell\le k_i} s_{i,\ell}.
\end{equation}
Suppose now that we have a summand of the form in (\ref{Aoperatorequation4}) obtained by applying
$\pa_{1j-1}^{(s_{j,\bullet})}$ to the only summand in (\ref{Aoperatorequation3}) satisfying {\it iii)}.
Since the $Y$-part remains unchanged, this implies already $t_{n,n}=s_{n,n},\ldots, t_{j+1,j+1}=s_{j+1,j+1}$.
Assume that we have already shown $t_{j,n}=s_{j,n},\ldots, t_{j,\ell_0+1}=s_{j,\ell_0+1}$, then we have to show
that  $t_{j,\ell_0}\le s_{j,\ell_0}$.

We consider five cases:
\begin{itemize}
\item $\ell_0 > k_{j}$. In this case the root $\al_{j,\ell_0}$ is not in the support of $\bp$ and hence
$s_{j,\ell_0}=0$. Since $\ell_0>k_{j}\ge k_{j-1}\ge\ldots\ge k_1$, for the same reason we have
$s_{i,\ell_0}=0$ for $i\le j$.
Recall that the divided power of $f^{(*)}_{1,\ell_0}$ in $A_{j-1}$ in (\ref{Ajoperator}) is equal to $s_{\bullet, \ell_0}$.
Now $s_{\bullet, \ell_0}=\sum_{i>j} s_{i,\ell_0}$ by the discussion above, and hence $f_{1,\ell_0}^{(s_{\bullet, \ell_0})}$
has already been transformed completely by the operators $\pa_{1,i}^{(*)}$, $i>j$,
and hence $t_{j,\ell_0}=0=s_{j,\ell_0}$.
\item $k_{j-1}<\ell_0\le k_{j}$. Since $\ell_0 > k_{j-1}\ge\ldots\ge k_1$, for the same reason as above
we have $s_{i,\ell_0}=0$ for $i<j$, so $s_{\bullet, \ell_0}=\sum_{i\ge j} s_{i,\ell_0}$.
The same arguments as above show that  for the operator $\pa_{1,j-1}^{(*)}$
only the power $f_{1,\ell_0}^{(s_{j,\ell_0})}$ is left to be transformed into a divided power of $f_{j,\ell_0}$,
so necessarily  $t_{j,\ell_0}\le s_{j,\ell_0}$.
\item $k_{j-1}=\ell_0=k_{j}$. In this case $s_{j,\bullet}=s_{j,\ell_0}$ and thus
the operator $\pa_{1,j-1}^{s_{j,\bullet}}=\pa_{1,j-1}^{s_{j,\ell_0}}$ can transform
a divided power $f^{(*)}_{1,\ell_0}$ in $A_j$ only into a power $f^{(q)}_{j,\ell_0}$ with $q$ at
most $s_{j,\ell_0}$.
\item $k_{j-1}= \ell_0 < k_{j}$. In this case $s_{j,\bullet}=s_{j ,\ell_0}+s_{j ,\ell_0+1}+\ldots+s_{j ,k_j}$.
Applying $\pa_{1,j-1}^{(s_{j,\bullet})}$ to the only summand in (\ref{Aoperatorequation3}) satisfying {\it iii)},
the assumption $t_{j,n}=s_{j,n},\ldots, t_{j,\ell_0+1}=s_{j,\ell_0+1}$ implies that one has to
apply $\pa_{1,j-1}^{(s_{j,k_j})}$ to $f_{1,k_j}^{(*)}$ and $\pa_{1,j-1}^{(s_{j,k_j-1})}$  to $f_{1,k_j-1}^{(*)}$ etc. to get the demanded
divided powers of the root vectors. So for $f^{(*)}_{1,\ell_0}$ only the operator $\pa_{1,j-1}^{(s_{j,\ell_0})}$ is left for transformations
into a divided power of $f_{j,\ell_0}$, and hence $t_{j,\ell_0}\le s_{j,\ell_0}$.
\item $\ell_0< k_{j-1}$. In this case $s_{j,\ell_0}=0$ because the root is not in the support.
Since $t_{j,\ell }=s_{j,\ell}$ for $\ell>\ell_0$ and $s_{j,\ell}=0$ for $\ell \le \ell_0$ (same reason as above)
we obtain
\[
\pa_{1,j-1}^{(s_{j,\bullet})}=\pa_{1,j-1}^{(\sum_{\ell>\ell_0} s_{j,\ell})}.
\]
But by assumption we know that $\pa_{1,j-1}^{(s_{j,\ell})}$ is needed to transform
the power  $f_{1,\ell}^{(s_{j,\ell})}$ into  $f_{j,\ell}^{(s_{j,\ell})}$ for all $\ell >\ell_0$,
so no divided power of $\pa_{1,j-1}$ is left and thus $t_{j,\ell_0}=0=s_{j,\ell_0}$.
\end{itemize}
It follows that all summands except one satisfy the conditions {\it i),ii)} above.
The only exception is the term where the divided powers of the operator $\pa_{1,j-1}^{(s_{j,\bullet})}$
are distributed as follows:
$$
\begin{array}{rl}
f_{1,1}^{(s_{\bullet, 1})}... f_{1,j-1}^{(s_{\bullet, j-1})}
(\pa_{1,j-1}^{(s_{j,j})}f_{1,j}^{(s_{\bullet, j})})& (\pa_{1,j-1}^{(s_{j,j+1})}f_{1,j+1}^{(s_{\bullet, j+1}-*)}) ...\\
&...(\pa_{1,j-1}^{(s_{j,n})} f_{1,n}^{(s_{\bullet, n}-*)}) f_{j+1,j+1}^{(s_{j+1,j+1})}... f_{n,n}^{(s_{n,n})}.
\end{array}
$$
By construction, this term has coefficient $1$ and satisfies the condition {\it iii)}, which finishes
the proof of the proposition.
\end{proof}

\begin{thm}\label{linearindipencetheorem}
The elements $f^{(\bs)} v_\lambda$ with $\bs\in S(\la)$ span the module $ V^a_\bz(\la)$.
\end{thm}
\begin{proof}
The elements $f^{(\bs)}$, $\bs$ arbitrary multi-exponent, span $S_\bz(\fn^{-,a})$, so the elements
$f^{(\bs)} v_\la$, $\bs$ arbitrary multi-exponent, span $S_\bz(\fn^{-,a})/I_\bz(\lam)\simeq V^a_\bz(\la)$.
We use now the equation \eqref{straighteninglawsl} in Proposition~\ref{straightening} as a
straightening algorithm to express $f^{(\bs)} v_\la $, $\bs$ arbitrary,
as a linear combination of elements $f^{(\bt)} v_\la$ such that $\bt\in S(\lam)$.

Let $\lam=\sum_{i=1}^n m_i\omega_i$ and suppose $\bs\notin S(\la)$, then there exists a Dyck path $\bp=(p(0),\dots,p(k))$
with $p(0)=\al_i$, $p(k)=\al_j$ such that
\[
\sum_{l=0}^k s_{p(l)} > m_i +\dots + m_j.
\]
We define a new multi-exponent $\bs'$ by setting
\[\bs'_\al=
\begin{cases}
s_\al, \ \al\in\bp,\\
0,\ otherwise.
\end{cases}
\]
For the new multi-exponent $\bs'$ we still have
$$
\sum_{l=0}^k s'_{p(l)} > m_i +\dots + m_j.
$$
We can now apply Proposition~\ref{straightening} to $\bs'$ and conclude
$$
f^{(\bs')} =\sum_{\bs' \succ \bt'} c_{\bt'} f^{(\bt')}\quad\text{in}\quad S_\bz(\fn^{-,a})/I_\bz(\lam),
$$
where $c_{\bt'}\in\bz$.
We get $f^{(\bs)}$ back as $f^{(\bs)}=f^{(\bs')}\prod_{\beta\notin\bp} f_\beta^{(s_\beta)}$.
For a multi-exponent $\bt'$ occurring in the sum with $c_{\bt'}\not=0$ let the multi-exponent
$\bt$ and $c_\bt\in\bz$ be such that
$c_{\bt'}f^{(\bt')}\prod_{\beta\notin\bp} f_\beta^{(s_\beta)}=c_\bt f^{(\bt)}$ (recall \eqref{symbolrelation}). Since we have a monomial
order it follows:
\begin{equation}\label{tprime3}
f^{(\bs)} =f^{(\bs')}\prod_{\beta\notin\bp} f_\beta^{(s_\beta)}=\sum_{\bs\succ \bt} c_\bt f^{(\bt)} \quad\text{in}\quad S_\bz(\fn^{-,a})/I_\bz(\lam).
\end{equation}
The equation \eqref{tprime3} provides an algorithm
to express $f^{(\bs)}$ in $S_\bz(\fn^{-,a})/I_\bz(\lam)$ as a sum of elements of the desired form: if some
of the $\bt$ are not elements of $S(\lam)$, then we
can repeat the procedure and express the $f^{(\bt)}$ in $S_\bz(\fn^{-,a})/I_\bz(\lam)$ as a sum of
$f^{(\br)}$ with $\br\prec\bt$. For the chosen ordering any strictly decreasing
sequence of multi-exponents (all of the same total degree) is finite, so after a finite number of steps one obtains an
expression of the form $f^{(\bs)}=\sum c_\br f^{(\br)}$ in $S_\bz(\fn^{-,a})/I_\bz(\lam)$ such that
$\br\in S(\lam)$ for all $\br$.
\end{proof}
\section{The main theorem for $SL_{n+1}$}
\begin{thm}\label{main}
The elements $\{f^{(\bs)} v_\lambda\mid \bs\in S(\la)\}$ form a basis for the module $V^a_\bz(\la)$
and the ideal $I_\bz(\lam)$ is generated by the subspace 
$$
\langle U_\bz(\n^+)\circ f_{\al_{i,j}}^{(m_i+\ldots+m_j+1)}\mid 1\le i\le j \le n-1\rangle.
$$
\end{thm}
As an immediate consequence we see:
\begin{cor}
\begin{itemize}
\item[{\it i)}] $V^a_\bz(\la)$ is a free $\bz$-module.
\item[{\it ii)}] For every $\bs\in S(\la)$ fix a total order
on the set of positive roots and denote by abuse of notation
by $f^{(\bs)}\in U_\bz(\n^-)$ also the corresponding product of divided powers.
The $\{f^{(\bs)} v_\lambda\mid \bs\in S(\la)\}$ form a basis for the module $V_\bz(\la)$
and for all $s<s'$ we have $V_\bz(\la)_s$ is a direct summand of $V_\bz(\la)_{s'}$ as a $\bz$-module.
\item[{\it iii)}] With the notation as above: let $k$ be a field and denote by 
$V_k(\la)=V_\bz(\la)\otimes_\bz k$, $U_k(\g)=U_\bz(\g)\otimes_\bz k$, 
$U_k(\n^-)=U_\bz(\n^-)\otimes_\bz k$ etc. the objects obtained by base change. 
The $\{f^{(\bs)} v_\lambda\mid \bs\in S(\la)\}$ form a basis for the module $V_k(\la)$.
\end{itemize}
\end{cor}
\proof
We know that the elements $f^{(\bs)} v_\la$, $\bs\in S(\la)$, span $V^a_\bz(\la)$, see
Theorem~\ref{linearindipencetheorem}. By [FFL], the number $\sharp S(\la)$
is equal to $\dim V(\lam)$, which implies the linear independence. By lifting the elements
to $V_\bz(\la)$, we get a basis of $V_\bz(\la)$ which is (by construction) compatible
with the PBW-filtration: set 
$$
S(\la)_r=\{\bs\in S(\la)\mid \sum_{\beta\in R^+} s_\beta\le r\},
$$ 
then the elements $f^{(\bs)} v_\la$, $\bs\in S(\la)_r$, span $V_\bz(\la)_r$.

Let $I\subset S_\bz(\fn^{-,a})$ be the ideal generated by 
$$
\langle U_\bz(\n^+)\circ f_{\al_{i,j}}^{(m_i+\ldots+m_j+1)}\mid 1\le i\le j \le n-1\rangle,
$$
by construction we know $I\subseteq I_\bz(\lam)$. But we also know 
that the relations in $I$ are sufficient to rewrite every element in $V^a_\bz(\la)$
in terms of the basis elements $f^{(\bs)} v_\lambda$, $\bs\in S(\la)$, which implies that the
canonical surjective map $S_\bz(\fn^-)/I\rightarrow S_\bz(\fn^-)/I_\bz(\la)\simeq V_\bz(\la)$ is injective.
\qed
\section{Symplectic Dyck paths}\label{DyckSP}
We recall the notion of the symplectic Dyck paths:
\begin{dfn}\label{dyckpath}\rm
A symplectic Dyck path (or simply a path) is a sequence
\[
\bp=(\beta(0), \beta(1),\dots, \beta(k)), \ k\ge 0
\]
of positive roots satisfying the following conditions:
\begin{itemize}
\item[{\it a)}] the first root is simple, $\beta(0)=\al_i$ for some $ 1 \leq i \leq n$;
\item[{\it b)}] the last root is either simple or the highest root of a symplectic subalgebra, more precisely $\beta(k) = \al_j$ or $\beta(k) = \al_{j \ol{j}}$ for some 
$ i \le j \leq n$;
\item[{\it c)}] the elements in between obey the following recursion rule:
If $\beta(s)=\al_{p,q}$ with $p, q \in J$ (see \eqref{sa}) then the next element in the sequence 
is of the form either
$\beta(s+1)=\al_{p,q+1}$  or $\beta(s+1)=\al_{p+1,q}$, where $x+1$ denotes the smallest element in $J$ which is bigger than $x$.
\end{itemize}
\end{dfn}

Denote by $\D$ the set of all Dyck paths.
For a dominant weight $\la=\sum_{i=1}^n m_i\omega_i$
let $P(\lam)\subset \R^{n^2}_{\ge 0}$ be the polytope
\begin{equation}
\label{polytopeequation}
P(\lam):=\bigg\{(s_ \al)_{\al> 0}\mid \forall\bp\in\D:
\begin{array}{l}
\text{ If }\beta(0)=\al_i,\beta(k)=\al_j, \text{ then }\\
s_{\beta(0)} + \dots + s_{\beta(k)}\le m_i + \dots + m_j,\\
\text{ if } \beta(0)=\al_i,\beta(k)=\al_{\ol{j}}, \text{ then }\\
s_{\beta(0)} + \dots + s_{\beta(k)}\le m_i + \dots + m_n\\
\end{array}\bigg\},
\end{equation}
and let $S(\la)$ be the set of integral points in $P(\lam)$.
\vskip 5pt\noindent
For a multi-exponent $\bs=\{s_ \beta\}_{\beta>0}$, $s_ \beta\in\Z_{\ge 0}$, let $f^{(\bs)}$ be the element
\[
f^{(\bs)}=\prod_{\beta\in R^+} f_ \beta^{(s_ \beta)}\in S_\bz(\fn^{-,a}).
\]

\section{The spanning property for the symplectic Lie algebra}\label{spanningsection}
Our aim is to prove that the set $f^{(\bs)} v_\la$, $\bs\in S(\la)$, forms a basis of $V_\bz^a(\la)$.
As a first step we will prove that these elements span $V_\bz^a(\la)$.

\begin{lem}
Let $\la=\sum_{i=1}^n m_i\omega_i$ be the $\msp_{2n}$-weight and let  $V_\bz(\la)\subset V(\la)$
be the corresponding lattice in the highest weight module with highest weight vector $v_\la$.
Then
\begin{gather}
f_{\al_{i,j}}^{(m_i+\dots +m_j+1)}v_\la=0,\ 1\le i\le j\le n-1,\label{1}\\
f_{\alpha_{i, \ol{i}}}^{(m_i+\dots + m_n+1)}v_\la=0,\ 1\le i\le n.\label{2}
\end{gather}
\end{lem}
\begin{proof}
The lemma follows immediately from the $\msl_2$-theory.
\end{proof}

In the following we use the operators $\pa^{(k)}_\al$ defined by
$\pa^{(k)}_\al(f^{(m)}_\beta)=0$ if $\al=\beta$ or
if the root vectors commute, and if $\alpha,\gamma,\beta=\alpha+\gamma$ are positive roots
spanning a subsystem of type ${\tt A}_2$, then 
\begin{equation} \label{dad1}
\pa^{(k)}_\al(f^{(m)}_\beta) =
\begin{cases}
 f_{\gamma}^{(k)}f^{(m-k)}_\beta,\  \text{ if } k\le m,\\
0,\ \text{ otherwise}.
\end{cases}
\end{equation}
If $\alpha,\gamma,\alpha+\gamma,\alpha+2\gamma$ span a subrootsystem of
type ${\tt B}_2={\tt C}_2$, then 
\begin{equation} \label{dad2}
\pa^{(k)}_\al(f^{(m)}_{\alpha+\gamma}) =
\begin{cases}
f_{\gamma}^{(k)}f^{(m-k)}_{\alpha+\gamma},\  \text{ if } k\le m,\\
0,\ \text{ otherwise},
\end{cases},
\end{equation}
and 
\begin{equation} \label{dad3}
\pa^{(k)}_{\al+\gamma}(f^{(m)}_{\alpha+2\gamma}) =
\begin{cases}
f_{\gamma}^{(k)}f^{(m-k)}_{\alpha+2\gamma},\  \text{ if } k\le m,\\
0,\ \text{ otherwise},
\end{cases},
\end{equation}
and 
\begin{equation} \label{dad4}
\pa^{(k)}_\gamma(f^{(m)}_{\alpha+\gamma}) =
\begin{cases}
 2^k f_{\alpha}^{(k)}f^{(m-k)}_{\alpha+\gamma},\  \text{ if } k\le m,\\
0,\ \text{ otherwise},
\end{cases}
\end{equation}
and 
\begin{equation} \label{dad5}
\pa^{(k)}_\gamma(f^{(m)}_{\alpha+2\gamma}) =
\begin{cases}
f_{\alpha+\gamma}^{(k)}f^{(m-k)}_{\alpha+2\gamma}\\
\hskip 15pt+\sum_{\stackrel{c>m-k}{a+b+c=k}}c_{a,b,c} f_{\alpha}^{(a)}f_{\alpha+\gamma}^{(b)}f^{(c)}_{\alpha+2\gamma},\  \text{ if } k\le m,\\
0,\ \text{ otherwise},
\end{cases}
\end{equation}
with the coefficients $c_{a,b,c} $ chosen such that 
$\pa^{(k)}_\gamma(f^{(m)}_{\alpha+2\gamma})=\pm  (\text{ad}\, e^{(k)}_\gamma(f^{(m)}_{\alpha+2\gamma}))$
Note that all the operators are such that $\pa^{(k)}_\gamma=\pm  (\text{ad}\, e^{(k)}_\gamma)$ (see \eqref{ad0}--\eqref{ad5}). 
In the following we sometimes use the equality $\al_{i,\ol{n}}=\al_{i,n}$.
\begin{lem}\label{pabeal}
The only non-trivial vectors of the form $\pa_\beta f_\al$, $\al,\beta>0$ are as follows:
for $\al=\al_{i,j}$, $1\le i\le j\le n$
\begin{equation}\label{sp1}
\pa_{i,s}f_{i,j}=f_{s+1,j},\ i\le s<j,\quad
\pa_{s,j}f_{i,j}=f_{i,s-1},\ i< s\le j,
\end{equation}
and for $\al=\al_{i,\ol{j}}$, $1\le i\le j\le n$
\begin{gather}\label{sp2}
\pa_{i,s}f_{i,\ol{j}}=f_{s+1,\ol{j}},\ i\le s <j,\quad
\pa_{i,s}f_{i,\ol{j}}=f_{j,\ol{s+1}},\ j\le s,\quad
\pa_{i,\ol{s}}f_{i,\ol{j}}=f_{j,s-1},\ j<s,\\
\label{sp3}
\pa_{s+1,\ol{j}}f_{i,\ol{j}}=f_{i,s},\ i\le s <j,\quad
\pa_{j,\ol{s+1}}f_{i,\ol{j}}=f_{i,s},\ j\le s,\quad
\pa_{j,s-1}f_{i,\ol{j}}=f_{i,\ol{s}},\ j<s.
\end{gather}
\end{lem}

Let us illustrate this lemma by the following picture in type ${\tt C}_5$.
\vskip 5pt
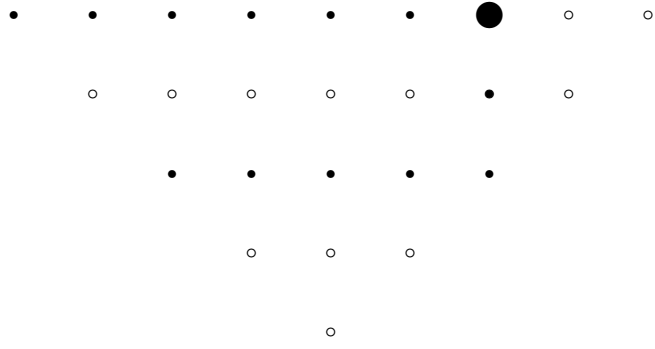
\[
\begin{picture}(200,150)
\multiput(0,150)(30,0){6}{\circle*{3}}
\put(180,150){\circle*{10}}
\multiput(210,150)(30,0){2}{\circle{3}}

\multiput(30,120)(30,0){7}{\circle{3}}
\put(180,120){\circle*{3}}

\multiput(60,90)(30,0){5}{\circle*{3}}

\multiput(90,60)(30,0){3}{\circle{3}}
\multiput(120,30)(30,0){1}{\circle{3}}
\end{picture}
\]

Here all circles correspond to the positive roots of the root system of type ${\tt C}_5$ 
in the following way:
in the upper row we have from left to right  $\al_{1,1}, \dots, \al_{1,5}, \al_{1,\ol{4}},\dots, \al_{1,\ol{1}}$,
in the second row we have from left to right $\al_{2,2}, \dots, \al_{2,5}, \al_{2,\ol{4}},\dots, \al_{2,\ol{2}}$,
and the last line corresponds to the root $\al_{5,5}$. Now let us take the root
$\al_{1,\ol{3}}$ (which corresponds to the fat circle). Then all roots that 
can be obtained by applying the operators $\pa_\beta$ are depicted as filled 
small circles.

\begin{thm}\label{spanCn}
\begin{itemize}
\item[{\it i)}] The vectors $f^{(\bs)} v_\la$, $\bs\in S(\la)$ span $V^a_\bz(\la)$.
\item[{\it ii)}] Let $I_\bz(\la)=S_\bz(\n^-)(\U_\bz(\n^+)\circ R)$, i.e. $I_\bz(\la)$ is generated
by $(\U_\bz(\n^+)\circ R)$, where
$$
R=\mathrm{span}\{ f_{\al_{i,j}}^{(m_i+\dots +m_j+1)}, 1\le i\le j\le n-1,\
f_{\alpha_{i, \ol{i}}}^{(m_i+\dots + m_n+1)}, 1\le i\le n\}.
$$ 
There exists an order `` $\ord$" on the ring $S_\bz(\n^{-,a})$ such that
for any $\bs\not\in S(\la)$ there exists a homogeneous expression 
(a straightening law) of the form
\begin{equation}\label{straighteninglaw}
f^{(\bs)} -\sum_{\bs\ord \bt} c_\bt f^{(\bt)} \in I_\bz(\la).
\end{equation}
\end{itemize}
\end{thm}

\begin{rem}
In the following we refer to $(\ref{straighteninglaw})$ as a {\it straightening law} for 
$S_\bz(\n^{-,a})$ with respect to the ideal $I_\bz(\la)$.
Such a straightening law implies that in the quotient ring $S_\bz(\n^{-,a})/I_\bz(\la)$ we can express
$f^{(\bs)}$ as a linear combination of monomials which are smaller in the order,
but of the same total degree since the expression in $(\ref{straighteninglaw})$ is homogeneous.
\end{rem}

First we show that {\it ii)} implies {\it i)}:
\begin{proof} {\bf [{\it ii)} $\Rightarrow ${\it i)}]} The elements in $R$ obviously
annihilate $v_\la\in V^a_\bz(\la)$, and so do the elements of $\U_\bz(\n^+)\circ R$, 
and hence so do the
elements of the ideal $I$ generated by $\U_\bz(\n^+)\circ R$. 
As a consequence we get a surjective map $S(\n^-)/I \rightarrow V^a_\bz(\la)$.

Suppose $\bs\not\in S(\la)$. We know by {\it ii)} that 
$f^{(\bs)} = \sum_{\bs\ord \bt} c_\bt f^{(\bt)}$
in $S_\bz(\n^{-,a})/I$. If any of the $\bt$ with nonzero coefficient $c_\bt$ is not an element in $S(\la)$,
then we can again apply a straightening law and replace $f^{(\bt)}$ by a linear combination of
smaller monomials. Since there are only a finite number of monomials
of the same total degree, by repeating the procedure if necessary, after a finite number of steps
we obtain an expression of $f^{(\bs)}$ in $S_\bz(\n^{-,a})/I$ as a linear combination of elements $f^{(\bt)}$, $\bt\in S(\la)$.
It follows that $\{ f^{(\bt)}\mid \bt\in S(\la)\}$ is a spanning set for $S_\bz(\n^{-,a})/I$, and hence, by the
surjection above, we get a spanning set $\{ f^{(\bt)} v_\la\mid \bt\in S(\la)\}$ for $V^a_\bz(\la)$.
\end{proof}

To prove the second part we need to define the total order. We start by defining a total
order on the variables:
\begin{equation}\label{order}
\begin{array}{rcl}
f_{1, {1}}<f_{1, 2}<\ldots <f_{1, {n-1}}<&f_{1,n}&<f_{1,\overline{n-1}}<\ldots<f_{1,\overline{2}}<f_{1,\overline{1}}\\
<\ldots<&\ldots&<\ldots <\\
<f_{n-2, n-2}<f_{n-2, {n-1}}<&f_{n-2, n}&<f_{n-2, \overline{n-1}}<f_{n-2, \overline{n-2}}\\
<f_{n-1,n-1}<&f_{n-1, n}&<f_{n-1, \overline{n-1}}\\
<&f_{n,n}&\\
\end{array},
\end{equation}
so, given an element $f_{x,y}$, the elements in the rows below and the elements 
on the right side in the same row are larger than $f_{x,y}$. 
\begin{rem}
If we omit in \eqref{order} above the elements $f_{i,\bar j}$, $i=1,\ldots,n$,
$i\le j\le n-1$, then we have the order in the case $\fg={\mathfrak{sl}}_n$.
\end{rem}
We use the same notation for the induced homogeneous lexicographic ordering on the monomials.
Note that this monomial order $>$ is not the order $\ord$ we define now.
Let 
\begin{gather*}
s_{\bullet, j}=\sum_{i=1}^j s_{i,j},\quad s_{\bullet, \ol{j}}=\sum_{i=1}^j s_{i,\ol{j}}, \\
s_{i,\bullet}=\sum_{j=i}^n s_{i,j} + \sum_{j=i}^{n-1} s_{i,\ol{j}}.
\end{gather*}
Define a map $d$ from the set of multi-exponents $\bs$ to $\Z_{\ge 0}^n$:
\[
d(\bs)=(s_{n,\bullet},s_{n-1,\bullet},\dots,s_{1,\bullet}).
\]
So, $d(\bs)_i=s_{n-i+1,\bullet}$.
We say $d(\bs)>d(\bt)$ if there exists an $i$ such that
$$
d(\bs)_1=d(\bt)_1, \dots, d(\bs)_i=d(\bt)_i, d(\bs)_{i+1} > d(\bt)_{i+1}.
$$
\begin{dfn}\label{sp<}
For two monomials $f^{(\bs)}$ and $f^{(\bt)}$  we say $f^{(\bs)} \ord f^{(\bt)}$ if 
\begin{itemize}
\item[{\it a)}] the total degree of $f^{(\bs)}$ is bigger than the total degree of $f^{(\bt)}$;
\item[{\it b)}] both have the same total degree but $d(\bs) < d(\bt)$;
\item[{\it c)}] both have the same total degree, $d(\bs) = d(\bt)$, but $f^{(\bs)} > f^{(\bt)}$.
\end{itemize}
\end{dfn}
In other words: if both have the same total degree, this definition says that $f^{(\bs)}$ is greater than $f^{(\bt)}$ if 
$d(\bs)$ is smaller than $d(\bt)$,  or $d(\bs)=d(\bt)$ but $f^{(\bs)} > f^{(\bt)}$
with respect to the homogeneous lexicographic ordering on $S_\bz(\n^-)$.

\begin{rem}\label{monorder}
It is easy to check that ``$\ord$" defines a ``{\it monomial ordering}'' in the following sense:
if $f^{(\bs)}\ord f^{(\bt)}$ and $f^{(\bm)}\not=1$, then
$$
f^{(\bs+\bm)}\ord   f^{(\bt+\bm)}\ord f^{(\bt)}.
$$
\end{rem}

By abuse of notation we use the same symbol also for the multi-exponents: we write $\bs \ord \bt$ if and only if
$f^{(\bs)} \ord f^{(\bt)}$.


\vskip 3pt\noindent
{\it Proof of Theorem  \ref{spanCn} {\it ii)}}.
Let $\bs$ be a multi-exponent violating some of the Dyck path conditions from the definition of 
$S(\la)$. As in the proof of Theorem~\ref{linearindipencetheorem}, 
it suffices to consider the case where $\bs\not\in S(\la)$ and $\bs$
is supported on a Dyck path $\bp$ and $\bs$ violates the Dyck path condition for
$S(\la)$ for this path $\bp$.

Suppose first that the Dyck path $\bp$ is such that $p(0)=\al_i$, $p(k)=\al_j$ for some $1\le i\le j< n$.
In this case the Dyck path involves only roots which belong to the Lie subalgebra 
${\mathfrak{sl}_n}\subset \mathfrak{sp}_{2n}$, and we get a straightening law by the results
in section~\ref{spansl}. By \eqref{Aoperatorequation} and Lemma~\ref{pabeal}, 
the application of the $\pa$-operators produces only summands such that $d(\bs)=d(\bt)$ for any $\bt$ occurring in the sum with
a nonzero coefficient. Hence we can replace ``$\succ$" by ``$\ord$" in $(\ref{straighteninglawsl})$, which finishes the proof 
of the theorem in this case.

Now assume $p(0)=\al_{i,i}$ and $p(k)=\al_{j,\ol{j}}$ for some $j\ge i$. 
We include the case $j=n$ by writing 
$\al_{n,n}=\al_{n,\ol{n}}$. We proceed by induction on $n$.
For $n=1$ we have ${\mathfrak{sp}}_{2}={\mathfrak{sl}}_{2}$, so we can refer 
to section~\ref{spansl}.
Now assume we have proved the existence of a straightening law for all symplectic algebras
of rank strictly smaller than $n$. If $i>1$, then the Dyck-path is also a Dyck-path for the symplectic subalgebra
$L\simeq \mathfrak{sp}_{2n-2(i-1)}$ generated by  $e_{\alpha_{k,k}}, f_{\alpha_{k,k}}, h_{\alpha_{k,k}}$, $i\le k \le n$. 
Let $\n^+_L, \n^-_L$ etc. be defined by the intersection of $\n^+,\n^-$ etc. with $L$ and set $\la_L=\sum_{k=i}^n m_k\om_k$.
It is now easy to see that the straightening law for $f^{(\bs)}$ viewed as an element in $S_\bz(\n^{-,a}_L)$
with respect to $I_{\bz,L}(\la_L)$ defines also a straightening law for $f^{(\bs)}$ viewed as an element in 
$S_\bz(\n^{-,a})$ with respect to $I_\bz(\la)$.

So from now on we fix $p(0)=\al_1$ and $p(k)=\al_{i,\ol{i}}$ for some $i\in\{1,\dots,n\}$. 
For a multi-exponent $\bs$ supported on $\bp$, set
$$
\Sigma=\sum_{l=0}^k s_{p(l)} > m_1+\dots +m_n.
$$
Obviously we have  $f_{1,\bar{1}}^{(\Sigma)}\in I(\lam)$. Now we consider two
operators
$$
\begin{array}{rl}
\Delta_1:=\pa_{1,i-1}^{(s_{\bullet,\bar i}+s_{i,\bullet})}
\underbrace{
\pa_{i+1,\ol{i+1}}^{(s_{\bullet,i})}\ldots\pa_{n,\bar n}^{(s_{\bullet,n-1})}}_{\delta_3} &
\underbrace{
\pa_{1,n-1}^{(s_{\bullet,n-1}+s_{\bullet,\ol{n}})} 
\ldots  
\pa_{1,i}^{(s_{\bullet,i}+s_{\bullet,\ol{i+1}})}
}_{\delta_2}\\
&\hskip 20pt \cdot \underbrace{
\pa_{1,\bar i}^{(s_{\bullet,i-1})}\dots \pa_{1,\bar 3}^{(s_{\bullet,2})}\pa_{1,\bar 2}^{(s_{\bullet,1})}
}_{\delta_1}
\end{array}
$$
and 
$$
\Delta_2:=\pa_{1,1}^{(s_{2,\bullet})}\pa_{1,2}^{(s_{3,\bullet})}\dots \pa_{1,i-2}^{(s_{i-1,\bullet})},
$$
and we will show that
\begin{equation}\label{goalstraightening}
\Delta_2\Delta_1 f_{1,\bar{1}}^{(\Sigma)}= f^{(\bs)} + \sum_{\bs\ord\bt} c_\bt f^{(\bt)} 
\end{equation}
with integral coefficients $c_\bt$. Since
$\Delta_2\Delta_1 f_{1,\bar{1}}^{(\Sigma)}\in I_\bz(\lam)$,  
the proof of $(\ref{goalstraightening})$ finishes the proof of the theorem.
A first step in the proof of $(\ref{goalstraightening})$ is the following lemma.

Recall the alphabet $J = \{1, \ldots, n, \ol{n-1}, \ldots, \ol{1}\}$.
Let $q_1,\dots,q_i\in J$ be a sequence of increasing elements defined by
\[
q_k=\max\{l\in J:\ \al_{k,l}\in\bp\}.
\]
For example, $q_i=\ol i$. All roots of $\bp$ are of the form
\[
\al_{1,1},\dots, \al_{1,q_1},\al_{2,q_1},\dots,\al_{2,q_2},\dots,
\al_{i,q_{i-1}},\dots,\al_{i,q_i}.
\]

\begin{lem}\label{B}
Set $f^{(\bs')}=f_{1,1}^{(s_{\bullet,1})}f_{1,2}^{(s_{\bullet,2})} \ldots 
f_{1,q_{i-1}}^{(s_{\bullet,q_{i-1}}-s_{i,q_{i-1}})} f_{i,q_{i-1}}^{(s_{i,q_{i-1}})}
\ldots f_{i,\bar i}^{(s_{i,\bar i})}$, then
\begin{equation}\label{equationin2.8}
\Delta_1 f_{1,\bar{1}}^{(\Sigma)}  =  f^{(\bs')} + \sum_{\bs'\ord\bt} c_\bt f^{(\bt)}. 
\end{equation}
If $f^{(\bt)}$, $\bt\not=\bs'$, 
is a monomial occurring in this sum, then 
either there exists an index $j$ such that 
$d(\bt)_j>0$ for some $j\in \{1,2,...,n-i\}$,
or $d(\bt)_j=0$ for all $j\in \{1,2,...,n-i\}$ and $d(\bt)_{n-j+1}>s_{i,\bullet}$,
or $d(\bt)=d(\bs')$ and 
$f_{i,i}^{(t_{i,i})}f_{i,i+1}^{(t_{i,i+1})}\cdots f_{i,\bar i}^{(t_{i,\bar i})}<
f_{i,i}^{(s_{i,i})}f_{i,i+1}^{(s_{i,i+1})}\cdots f_{i,\bar i}^{(s_{i,\bar i})}$.
\end{lem} 
\begin{cor}\label{unimportantmonomial}
If $f^\bt\not= f^{\bs'}$ is a monomial occurring in $(\ref{equationin2.8})$,
then either $\Delta_2 f^\bt=0$,  or $\Delta_2 f^\bt$ 
is a sum of monomials $f^\bk$ such that $f^\bs \ord f^\bk $.
\end{cor}
\vskip 4pt\noindent
{\it Proof of the lemma.\/}
One easily sees by induction that
$$
\delta_1(f_{1,\bar{1}}^{(\Sigma)})=f_{1,1}^{(s_{\bullet,1})}f_{1,2}^{(s_{\bullet,2})}
\ldots f_{1,i-1}^{(s_{\bullet,i-1})}f_{1,\bar{1}}^{(\Sigma-s_{\bullet,1}-s_{\bullet,2}-\ldots-s_{\bullet,i-1})}.
$$
Note that the roots used in the operator are $\epsilon_1+\epsilon_2,\ldots,\epsilon_1+\epsilon_i$, and they are 
applied to $f_{1,\bar 1}$ of weight $2\epsilon_1$. In terms of \eqref{ad2}--\eqref{ad5}, we apply
$\pa^{(*)}_{\alpha+\gamma}$ to $f^{(*)}_{\alpha+2\gamma}$, so rule \eqref{ad3} applies.

Since $\al_{1,j}-\al_{1,\ell}$, $1\le j<i$, $i < \ell \le n$, and $\al_{1,j}-\al_{\ell,\bar\ell}$, 
$1\le j<i$, $i < \ell \le n$, and $\al_{1,j}-\al_{1,i-1}$, $1\le j<i$, 
are never positive roots, one has
$$
\pa_{1,i-1}^{(s_{\bullet,\bar i}+s_{i,\bullet})}\delta_3\delta_2(
\underbrace{f_{1,1}^{(s_{\bullet,1})}f_{1,2}^{(s_{\bullet,2})}\ldots f_{1,i-1}^{(s_{\bullet,i-1})}}_{f^{(\bx)}})=0,
$$
so it remains to consider 
$f^{(\bx)}\pa_{1,i-1}^{(s_{\bullet,\bar i}+s_{i, \bullet})}
\delta_3\delta_2(f_{1,\bar{1}}^{(\Sigma-s_{\bullet,1}-s_{\bullet,2}-\ldots-s_{\bullet,i-1})})$.

To better visualize the following procedure, one should think
of the variables $f_{i,j}$ as being arranged in a triangle like in the picture after Lemma~\ref{pabeal},
or in the following example (type ${\tt C}_4$):
\begin{equation}\label{scheme}
\Skew(0: f_{11},f_{12} ,f_{13} ,f_{14} ,f_{1\bar{3}},f_{1\bar{2}},f_{1\bar{1}} |
1:f_{22} ,f_{23} ,f_{24} ,f_{2\bar{3}},f_{2\bar{2}} |2:f_{33} ,f_{34} ,f_{3\bar{3}}|3:f_{44} )
\end{equation}
With respect to the ordering ``$>$", the largest element is located in the bottom row and the 
smallest element is written in the top row on the left side. We enumerate the rows and 
columns like the indices of the variables, so the top row is the 1-st row, the bottom 
row the $n$-th row, the columns are enumerated from the left to the right,
so we have the $1$-st column on the left side and the most right one is the $\bar{1}$-st column. 

The operator $\pa_{1,q}$, $1\le q\le n-1$, kills all $f_{1,j}$ for $1\le j\le q$,
$\pa_{1,q}(f_{1,j})=f_{q+1,j}$ for $j=q+1,\ldots,\ol{q+1}$ ({{rule \eqref{ad1}} applies}),
$\pa_{1,q}(f_{1,\bar j})=f_{j, \ol{q+1}}$ for $j=1,\ldots, q$ ({{rule \eqref{ad1}} applies}), 
and $\pa_{1,q}$ kills all $f_{k,\ell}$ for $k\ge 2$. Because of the set of indices of the 
operators occurring in $\delta_2$, the operator applied to 
$f_{1,\bar{1}}^{(\Sigma-s_{\bullet,1}-s_{\bullet,2}-\ldots-s_{\bullet,i-1})}$ never 
increases the zero entries in the first row, column $\bar i$ up to column $\bar 2$.
As a consequence, the application of $\delta_2$ 
produces  the sum of monomials
$$
f^{(\bx)} f_{1,\ol{i+1}}^{(s_{\bullet,i}+s_{\bullet,\ol{i+1}})}\cdots f_{1,\ol{n-1}}^{(s_{\bullet,n-2}+s_{\bullet,\ol{n-1}})} 
f_{1,n}^{(s_{\bullet,n-1}+s_{\bullet,n})} f_{1,\bar 1}^{(s_{\bullet,\bar i})}+\sum c_\bk f^{(\bk)},
$$
where the monomials $f^{(\bk)}$ occurring in the sum are such that the corresponding 
triangle (see $(\ref{scheme})$) has at least one 
non-zero entry in one of the rows between the $(i+1)$-th row and the $n$-th row 
(counted from top to bottom). 
This implies $d(\bk)_j>0$ for some $j=1,\ldots, n-i$. The operators 
$\delta_3$ and $\pa_{1,i-1}^{(s_{\bullet,\bar i}+s_{i, \bullet})}$ do not change this property
because (in the language of the scheme $(\ref{scheme})$ above) the
operators $\pa_{j,\bar j}$ used to compose $\delta_3$ either kill a monomial or, in the 
language of the scheme $(\ref{scheme})$, they subtract from an entry in the $\bar j$-th column, $k$-th row
and add to the entry in the same row, but $(j-1)$-th column.
The operator $\pa_{1,i-1}$
subtracts from the entries in the top row and, since the 
entries in the top row, column $\ol{i-1}$ up to $\bar 2$ are zero, 
adds to the entries in the $i$-th row. The only exception is $\pa_{1,i-1}$ applied to $f_{1,\bar 1}$,
the result is $f_{1,\bar i}$. 
It follows that the monomials $f^{(\bk')}$ 
occurring in $\pa_{1,i-1}^{(s_{\bullet,\bar i}+s_{i, \bullet})}\delta_3 f^{(\bk)}$
have already the desired properties because we have just seen that 
$d(\bk')_j>0$ for some $j=1,\ldots, n-i$.

So to finish the proof of the lemma, in the following it suffices to consider 
\begin{equation}\label{whoknows}
\begin{array}{l}
f^\bx \pa_{1,i-1}^{(s_{\bullet,\bar i}+s_{i, \bullet})} \delta_3 f_{1,\ol{i+1}}^{(s_{\bullet,i}+
s_{\bullet,\ol{i+1}})}\cdots f_{1,\ol{n-1}}^{(s_{\bullet,n-2}+s_{\bullet,\ol{n-1}})} 
f_{1,n}^{(s_{\bullet,n-1}+s_{\bullet,n})} f_{1,\bar 1}^{(s_{\bullet,\bar i})} \\ =
f^\bx \pa_{1,i-1}^{(s_{\bullet,\bar i}+s_{i, \bullet})}  f_{1,i}^{(s_{\bullet,i})}f_{1,i+1}^{(s_{\bullet,i+1})}
\cdots f_{1,n}^{(s_{\bullet,n})} 
f_{1,\ol{n-1}}^{(s_{\bullet,\ol{n-1}})}\cdots 
 f_{1,\ol{i+1}}^{(s_{\bullet,\ol{i+1}})}f_{1,\bar 1}^{(s_{\bullet,\bar i})}.
\end{array}
\end{equation}
Note that the operators in $\delta_3$ are of the form $\pa_{j,\bar j}$, 
$j=i+1,\ldots,n$, and they are applied to $f_{1,\bar\ell}$, $\ell=i+1,\ldots,n$, 
so $\pa^{(k)}_{j,\bar j}f_{1,\bar\ell}^{(p)}=0$ for $\ell\not=j$
and for $j=\ell$ we set $\alpha=2\epsilon_j$, $\gamma=\epsilon_1-\epsilon_j$,
$\pa_{j,\bar j}=\pa_{\alpha}$, $f_{1,\bar j}=f_{\alpha+\gamma}$, so
rule \eqref{ad2} applies and the coefficient in \eqref{whoknows} is 1.

To apply $\pa_{1,i-1}$ to the monomial above increases in each step the
degree with respect to the variables $f_{i,*}$, unless the operator is
applied to a variable killed by the operator or to $f_{1,\bar 1}$,
in which case the result is $f_{1,\bar i}$ (note that in this case
rule \eqref{ad3} applies).
So the right hand side of $(\ref{whoknows})$
can be written as a linear combination $\sum c_\bk f^{(\bk)}$ of monomials
such that $d(\bk)_{j}=0$ for $j=1,\ldots,n-i$ and $d(\bk)_{n-i+1}\ge s_{i, \bullet}$.

It remains to consider the case where 
$d(\bk)_{n-i+1}= s_{i, \bullet}$. This is only possible
if $\pa_{1,i-1}$ is applied  
$s_{\bullet,\bar i}$-times to $f_{1,\bar 1}^{s_{\bullet,\bar i}}$,
in which case $d(\bk)$ has only two non-zero entries: $d(\bk)_1=\Sigma-s_{i, \bullet}$
and $d(\bk)_{n-i+1}=s_{i, \bullet}$, so $d(\bk)=d(\bs')$. If $\bk\not=\bs'$,
then necessarily
$f_{i,i}^{(t_{i,i})}f_{i,i+1}^{(t_{i,i+1})}\cdots f_{i,\bar i}^{(t_{i,\bar i})}<
f_{i,i}^{(s_{i,i})}f_{i,i+1}^{(s_{i,i+1})}\cdots f_{i,\bar i}^{(s_{i,\bar i})}$. 
\qed

\vskip 4pt\noindent
{\it Proof of the corollary.\/}
The operators used to compose $\Delta_2$ do not change anymore the entries
of $d(\bt)$ for the first $n-i+1$ indices. 

Suppose first $\bt$ is such that there exists an index $j$ such that 
$d(\bt)_j>0$ for some $j\in \{1,2,...,n-i\}$ or $d(\bt)_{i,\bar i}>s_{i,\bullet}$. 
By the description of the operators
occurring in $\Delta_2$, every monomial $f^{(\bk)}$ occurring with a nonzero
coefficient in $\Delta_2 f^{(\bt)}$ has this property too and hence $f^{(\bs)} \ord f^{(\bk)} $.

Next assume $d(\bt)=d(\bs')$ and 
$f_{i,i}^{(t_{i,i})}f_{i,i+1}^{(t_{i,i+1})}\cdots f_{i,\bar i}^{(t_{i,\bar i})}<
f_{i,i}^{(s_{i,i})}f_{i,i+1}^{(s_{i,i+1})}\cdots f_{i,\bar i}^{(s_{i,\bar i})}$. 
Recall that $\bt_{1,\ol{i-1}}=\ldots=\bt_{1,\ol{1}}=0$. It follows that the operators
occurring in $\Delta_2$ always only subtract from one of the entries in the top row
and add to the entry in the same column and a corresponding row (of index strictly
smaller than $i$). It follows
that all monomials $f^{(\bk)}$ occurring in $\Delta_2(f^{(\bt)})$ have the property:
$d(\bk)=d(\bs)$. Since $f_{i,i}^{(t_{i,i})}f_{i,i+1}^{(t_{i,i+1})}\cdots f_{i,\bar i}^{(t_{i,\bar i})}<
f_{i,i}^{(s_{i,i})}f_{i,i+1}^{(s_{i,i+1})}\cdots f_{i,\bar i}^{(s_{i,\bar i})}$, 
it follows that $f^{(\bs)}> f^{(\bk)}$ and hence $f^{(\bs)}\ord f^{(\bk)}$. 
\qed
\vskip 4pt\noindent
{\it Continuation of the proof of Theorem~\ref{spanCn} ii).\/}
We have seen that, in order to prove Theorem~\ref{spanCn} {\it ii)}, it suffices to prove 
$(\ref{goalstraightening})$.
By Lemma~\ref{B} and Corollary~\ref{unimportantmonomial}, it remains to prove for
$f^{(\bs')}$ that 
$\Delta_2 f^{(\bs')}$ is a linear combination of  $f^{(\bs)}$ with coefficient 1
and monomials strictly smaller than $f^{(\bs)}$. The following lemma
proves this claim and hence finishes the proof of the theorem.
\qed

\vskip 4pt
The following lemma completes the proof of  part {\it ii)} of Theorem~\ref{spanCn}.
\begin{lem}\label{straightlemma}
The operator $\Delta_2:=\pa_{1,1}^{(s_{2,\bullet})}\pa_{1,2}^{(s_{3,\bullet})}\dots 
\pa_{1,i-2}^{(s_{i-1,\bullet})}$ applied to the monomial
$f^{(\bs')}$ is a linear combination of $f^{(\bs)}$ and smaller
monomials:
\begin{equation}\label{BC}
\Delta_2 f^{(\bs')} = f^{(\bs)} +\sum_{\bs \ord \bt} c_\bt f^{(\bt)}.
\end{equation}
\end{lem}
\begin{proof}
First note that all monomials $f^{(\bk)}$ occurring in 
$\Delta_2 f^{(\bs')}$ have the same total degree.
Recall that $\bs'_{1,\ol{i-1}}=\ldots=\bs'_{1,\ol{1}}=0$. It follows that the operators
occurring in $\Delta_2$ always only subtract from one of the entries in the top row
and add to the entry in the same column and a corresponding row (of index strictly
smaller than $i$ and strictly greater than $1$). It follows
that all monomials $f^{(\bk)}$ occurring in $\Delta_2(f^{(\bs')})$ have the same multidegree
$d(\bs)$, in fact, we will see below that $f^\bs$ is a summand and hence $d(\bk)=d(\bs)$. 

So in the following we can replace the ordering $\ord$ by $>$ since, in this special case,
the latter implies the first.

The elements $f_{i,j}$ and $f_{i,\bar{j}}$, $2\le i\le j \le n$, 
are in the kernel of the operators $\pa_{1,k}$ for all $1\le k\le n$, and 
so are the variables $f_{1,j}$, $j\le k$ in the first $k$ columns.

The operator $\pa_{1,k}$, $1\le k\le n$, ``moves" the variables $f_{1,j}$, $k+1\le j\le n$ from the first row to the variable $f_{k+1,j}$ in the 
same column, in this case rule \eqref{ad1} applies.

The operator $\pa_{1,k}$, $1\le k\le n$ ``moves" the variables $f_{1,\bar{j}}$, $k+1\le  j\le n$ from the 
first row to the variable $f_{k+1,\bar{j}}$ in the 
same column. Note that here rule \eqref{ad1} applies,
except for $j=k+1$, in this case set rule \eqref{ad2} applies.

For $j\le k$, the operator makes the variables switch the column,
it moves the variable $f_{1,\bar{j}}$ to the variable $f_{j, \overline{k+1}}$ in the $j$-th row and $(\overline{k+1})$-th column.
In this situation rule \eqref{ad1} applies, except if $j=1$. But note that $j=1$ can be excluded in our
case because $j=1$ implies $i=1$ for the path, and this implies that $\Delta_2$ is the identity operator, so there is no 
operator $\pa_{1,k}$ in this case.

We proceed by induction on $i$. If $i=1,2$, then $\Delta_2$ is the identity operator, 
$f^{(\bs)}=f^{(\bs')}$ and hence the lemma is trivially true.
Now assume $i\ge 3$ and the lemma holds for all numbers less than $i$.
We note that the monomial
\begin{gather*}
f_{1,1}^{(s_{1,1})}\dots f_{1,q_1}^{(s_{1,q_1})}\cdot
(\pa_{1,1}^{(s_{2,q_1})}f_{1,q_1}^{(s_{2,q_1})}\dots 
\pa_{1,1}^{(s_{2,q_2})}f_{1,q_2}^{(s_{2,q_2})})
\cdot \ldots \hskip 80pt \\
\ldots \cdot(\pa_{1,i-2}^{(s_{i -1,q_{i-2}})}f_{1,q_{i-2}}^{(s_{i-1,q_{i-2}})}\dots 
\pa_{1,i-2}^{(s_{i-1,q_{i-1}})}f_{1,q_{i-1}}^{(s_{i-1,q_{i-1}})})
(f_{i,q_{i-1}}^{(s_{i,q_{i-1}})}
\ldots f_{i,\bar i}^{(s_{i,\bar i})})
\end{gather*}
is equal to $f^\bs$ (only the rules \eqref{ad1} and \eqref{ad2} apply) and appears as a summand in $\Delta_2 f^{(\bs')}$.
Our goal is to show that all other monomials in $\Delta_2 f^{(\bs')}$ are less than $f^{(\bs)}$.

All monomials share the common factor $(f_{i,q_{i-1}}^{(s_{i,q_{i-1}})}
\ldots f_{i,\bar i}^{(s_{i,\bar i})})$, the maximal variable smaller than the ones occurring
in the divisor is the variable $f_{i-1,q_{i-1}}$.
Note that if $j<i-1$ then for any $q\in J$ the variable $\pa_{1,j}f_{1,q}$ 
lies in the $(j+1)$-th row, note that $j+1<i$. The operator $\pa_{1,i-2}$
is applied $s_{i-1,\bullet}$-times, the unique maximal monomial 
in the sum expression of $\pa_{1,i-2}^{(s_{i-1,\bullet})}f^{(\bs')}$ is
$$
f_{1,1}^{(s_{\bullet,1})}f_{1,2}^{(s_{\bullet,2})} \ldots 
f_{1,q_{i-2}}^{(s_{\bullet,q_{i-2}}- s_{i-1,q_{i-2}})}
(f_{i-1,q_{i-2}}^{(s_{i-1,q_{i-2}})} \ldots f_{i-1,q_{i-1}}^{(s_{i-1,q_{i-1}})})
(f_{i,q_{i-1}}^{(s_{i,q_{i-1}})} \ldots f_{i,\bar i}^{(s_{i,\bar i})}),
$$
because applying the operator $\pa_{1,i-2}$ to any of the variables $f_{1,j}$
such that $j\not=q_{i-2},\ldots, q_{i-1}$, gives a monomial smaller in the order $>$,
and the exponents $s_{i-1,j}$, $j=q_{i-2},\ldots,q_{i-1}$, are the maximal powers
such that $\pa^{(*)}_{1,i-2}$ can be applied to $f_{1,j}^{(y)}$ because either
$q_{i-2}<j<q_{i-1}$, and then $y=s_{\bullet,j}=s_{i-1,j}$, or $j=q_{i-1}$,
then  $s_{i-1,q_{i-1}}$ is the power with which the variable occurs in $f^{(\bs')}$, 
or $j=q_{i-2}$, then only the power $s_{i-1,q_{i-2}}$ of the operator is left.

Repeating the arguments for the operators  $\pa_{1,i-3}$ etc. finishes the proof
of the lemma.
\end{proof}
\section{The tensor product property}\label{tensors}
In the following section let $\g=SL_{n}$ or $Sp_{2n}$.
\begin{prop}
For two dominant weights $\la$ and $\mu$ the $S_\bz(\fn^{-,a})$-module
$V^a_{\bz}(\la+\mu)$ is embedded into the tensor product $V^a_{\bz}(\la)\T_\bz V^a_{\bz}(\mu)$ as the highest weight component, i.e.
there exists a unique injective homomorphism of $S_\bz(\fn^{-,a})$-modules: 
\begin{equation}\label{Cartan}
V^a_{\bz}(\la+\mu)\hk V^a_{\bz}(\la)\T V^a_{\bz}(\mu)\text{\ such that\ } v_{\la+\mu}\mapsto v_\la\T v_\mu. 
\end{equation}
\end{prop}
\begin{proof}
Using the defining relations for $V^a_{\bz}(\la+\mu)$, it is easy to see that we 
have a canonical  map $V^a_{\bz}(\la+\mu)\rightarrow V^a_{\bz}(\la)\otimes V^a_{\bz}(\mu)$
sending $v_\la$ to $v_{\la-\om_i}\T v_{\om_i}$. We know that $V^a_{\bz}(\la)\subset V^a(\la)$ and
$V^a_{\bz}(\mu)\subset V^a(\mu)$ are lattices in the corresponding complex vector spaces,
and, by \cite{FFoL1} and \cite{FFoL2}, we know that $S(\fn^{-,a})(v_{\la}\otimes v_{\mu})\subset V^a(\la)\otimes V^a(\mu)$
is isomorphic to $V^a(\lam+\mu)$, the isomorphism being given by
$$
V^a(\lam+\mu)\ni m.v_{\la+\mu}\mapsto  m.v_{\la}\otimes v_{\mu}\in   V^a(\la)\otimes V^a(\mu) \quad\text{for $m\in S(\fn^{-,a})$}.
$$
It follows that the induced map  $V^a_{\bz}(\la+\mu)\rightarrow V^a_{\bz}(\la)\otimes V^a_{\bz}(\mu)$
between the lattices is injective and hence an isomorphism onto the image.
\end{proof}
\section*{Acknowledgements}
The work of Evgeny Feigin was partially supported
by the Russian President Grant MK-3312.2012.1, by the Dynasty Foundation and 
by the AG Laboratory HSE, RF government grant, ag. 11.G34.31.0023.
This study comprises research findings from the `Representation Theory
in Geometry and in Mathematical Physics' carried out within The
National Research University Higher School of Economics' Academic Fund Program
in 2012, grant No 12-05-0014.
This study was carried out within "The National Research University Higher School of Economics' Academic Fund Program in 2012-2013, research grant No. 11-01-0017.
The work of Ghislain Fourier and Peter Littelmann was partially supported by the
priority program SPP 1388 of the German Science Foundation.


\begin{thebibliography}{99}
\bibitem[B]{B}
N.~Bourbaki, \emph{\'{E}l\'ements de math\'ematique. {F}asc. {XXXIV}. {G}roupes
et alg\`ebres de {L}ie. {C}hapitres {IV}, {V}, {VI}}, Actualit\'es
Scientifiques et Industrielles, No. 1337, Hermann, Paris, 1968.

\bibitem[Br]{Br}
R-K. Brylinski, {\it Limits of weight spaces, Lusztig's q-analogs and fiberings of
adjoint orbits,} J. Amer. Math. Soc, 2, no.3 (1989), 517-533.

\bibitem[F1]{F1}
E.~Feigin, {\it The PBW-filtration}, Represent. Theory 13  (2009), 165-181.

\bibitem[F2]{F2}
E.~Feigin,
{\it The PBW-Filtration, Demazure Modules and Toroidal Current Algebras},
SIGMA 4 (2008), 070, 21 pages.

\bibitem[FFJMT]{FFJMT}
B.~Feigin, E.~Feigin, M.~Jimbo, T.~Miwa, Y.~Takeyama,
{\it A $\phi_{1,3}$-filtration on the Virasoro minimal series $M(p,p')$
with $1<p'/p<2$},
Publ. Res. Inst. Math. Sci.  44  (2008),  no. 2, 213--257.

\bibitem[FFL1]{FFoL1}
E.~Feigin, G.~Fourier, P.~Littelmann,
{\it PBW-filtration and bases for irreducible modules in type $A_n$},
Transformation Groups {\bf 16}, Number 1 (2011), 71--89.

\bibitem[FFL2]{FFoL2}
E.~Feigin, G.~Fourier, P.~Littelmann,
{\it PBW-filtration and bases for symplectic Lie algebras},
International Mathematics Research Notices 2011; doi: 10.1093/imrn/rnr014.

\bibitem[FFL]{FFL}
B.~Feigin, E.~Feigin, P. Littelmann,
{\it Zhu's algebras, $C_2$-algebras and abelian radicals}, arXiv:0907.3962 (2009).


\bibitem[FF]{FF}
E.Feigin and M.Finkelberg, {\it Degenerate flag varieties of type A: Frobenius splitting and BWB theorem},
arXiv:1103.1491.

\bibitem[FFL]{FFiL}
E.Feigin, M.Finkelberg, P.Littelmann, {\it Symplectic degenerate flag varieties}, arXiv:1106.1399

\bibitem[FL]{FL}
E.~Feigin, P. Littelmann,
{\it Zhu's algebras, $C_2$-algebras and abelian radicals}, arXiv:0907.3962 (2009).


\bibitem[FH]{FH}
W.~Fulton, J.~Harris, {\it Representation Theory}, Graduate Texts in Mathematics, Springer Verlag,
New York 1991.


\bibitem[GG]{GG}
M.~R.~Gaberdiel, T.~Gannon, {\it Zhu's algebra, the $C_2$ algebra, and twisted modules},
arXiv:0811.3892

\bibitem[H]{H}
J.E. Humphreys,
{\it{Introduction to Lie algebras and representation Theory}}. Graduate Texts in Math., vol. 9,
Springer -Verlag (1970).

\bibitem[HJ]{HJ}
I. Heckenberger, A. Joseph
{\it On the left and right Brylinski-Kostant filtrations},
Algebr. Represent. Theory 12 (2009), no. 2-5, 417--442. 


\bibitem[J]{J}
J. C. Jantzen, {\it Representations of algebraic groups}, Pure and Applied Mathematics vol. 131, 
Academic Press, Orlando, 1987, xiii + 443 pp.

\bibitem[K]{K} B. Kostant, {\it Lie groups representations on polynomial rings,}
Amer. J. Math, 85, 327-404 (1963).

\bibitem[St]{Ste68} R. Steinberg, {\it Lectures on Chevalley groups}, Yale University, New Haven, Conn., 1968.
Notes prepared by John Faulkner and Robert Wilson.

\bibitem[T]{Tit87} J. Tits, {\it Uniqueness and presentation of Kac-Moody groups over fields}, J. Algebra,
105(2):542--573, (1987).

\end{thebibliography}
\end{document}